\documentclass[12pt]{amsart}
\usepackage{amsmath,amssymb,amsfonts,latexsym}
\usepackage{mathrsfs}%pour les caracteres en calligraphie
\usepackage[dvips]{geometry}
\usepackage{array}
\usepackage{epsf}
\usepackage{epsfig}
\newtheorem*{lemma}{Lemma}
\newtheorem*{prop}{Proposition}
\newtheorem*{thm}{Theorem}
\newtheorem*{cor}{Corollary}
\newcommand{\iso}{\overset{\sim}{\rightarrow}}

\newcommand{\nc}{\newcommand}
\nc{\ad}{\operatorname{ad}} \nc{\tr}{\operatorname{tr}}
\nc{\tp}{\operatorname{top}} \nc{\rank}{\operatorname{rank}}
\nc{\Inj}{\operatorname{Inj}} \nc{\Hom}{\operatorname{Hom}}
\nc{\End}{\operatorname{End}} \nc{\supp}{\operatorname{supp}}
\nc{\ch}{\operatorname{ch}} \nc{\im}{\operatorname{im}}
\nc{\gr}{\operatorname{gr}}\nc{\Id}{\operatorname{Id}}

\begin{document}

\title[Zhelobenko Invariants]{Zhelobenko Invariants, Bernstein-Gelfand-Gelfand operators and the analogue Kostant Clifford Algebra Conjecture}
\author[Anthony Joseph]{Anthony Joseph}

\
\footnotetext[1]{Work supported in part by
Israel Science Foundation Grant, no. 710724.}
\date{\today}
\maketitle

\vspace{-.9cm}\begin{center}
Donald Frey Professional Chair\\
Department of Mathematics\\
The Weizmann Institute of Science\\
Rehovot, 76100, Israel\\
anthony.joseph@weizmann.ac.il
\end{center}\



 \

\textbf{Abstract}

\

Let $\mathfrak g$ be a complex simple Lie algebra and $\mathfrak h$ a Cartan subalgebra. The Clifford algebra $C(\mathfrak g)$ of $\mathfrak g$ admits a Harish-Chandra map.  Kostant conjectured (as communicated to Bazlov in about 1997) that the value of this map on a (suitably chosen) fundamental invariant of degree $2m+1$ is just the zero weight vector of the simple $(2m+1)$-dimensional module of the principal s-triple obtained from the Langlands dual $\mathfrak g^\vee$.  Bazlov \cite {B1} settled this conjecture positively in type $A$.

The Kostant conjecture was reformulated (Alekseev-Bazlov-Rohr \cite {AM,B2,R}) in terms of the Harish-Chandra map for the enveloping algebra $U(\mathfrak g)$ composed with evaluation at the half sum $\rho$ of the positive roots.

Here an analogue of the Kostant conjecture is settled by replacing the Harish-Chandra map by a ``generalized Harish-Chandra" map which had been studied notably by Zhelobenko \cite {Z}.  The proof involves a symmetric algebra version of the Kostant conjecture (settled in works of Alekseev-Bazlov-Rohr), the Zhelobenko invariants in the adjoint case and surprisingly the Bernstein-Gelfand-Gelfand operators introduced in their study \cite {BGG} of the cohomology of the flag variety.

\section{Introduction}

\

The base field $k$ is assumed algebraically closed of characteristic zero throughout.

\subsection{}\label{1.1}

Let $\mathfrak g$ be a simple Lie algebra and $\mathfrak h$ Cartan subalgebra of $\mathfrak g$. Let $U(\mathfrak g)$ (resp. $S(\mathfrak g)$) denote the enveloping (resp. symmetric) algebra of $\mathfrak g$ with $Z(\mathfrak g)$ (resp. $Y(\mathfrak g)$) the corresponding invariant subalgebra for adjoint action.  Let $\mathscr F$ be the canonical (resp. degree) filtration on $U(\mathfrak g)$ (resp. $S(\mathfrak g)$.
%and $Z(\mathfrak g)$ (resp. $Y(\mathfrak g)$) the corresponding algebras of invariants under adjoint action.

Let $\Delta \subset \mathfrak h^*$ be the set of non-zero roots of $\mathfrak g$, $\Delta^+$ a choice of positive roots and $\pi=\{\alpha\}_{i\in I}:I:=1,2,\ldots,\ell$, the corresponding set of simple roots.

Given $\gamma \in \Delta$, let $\gamma^\vee$ denote the corresponding coroot (for example identified with $2\gamma/(\gamma,\gamma)$ through the Cartan inner product on $\mathfrak h^*$).  Let $\Delta^\vee$ (resp. $\Delta^{\vee+}$) denote the set of coroots (resp. positive coroots).

Let $s_i$ be the simple reflection defined by $\alpha_i\in \pi$ and $W$ the group they generate.  Let $\rho$ denote the half sum of the elements of $\Delta^+$.

For all $w \in W$, let $\ell(w)$ denote its reduced length relative to the given set of simple reflections.  For all $\gamma \in \Delta^+$ (resp. $\gamma \in \Delta^{\vee+}$), let $o(\gamma)$ (respect $o(\gamma^\vee)$) denote the sum of its coefficients with respect to the given set of simple roots (coroots).

Observe that $o(\gamma^\vee)=\rho(\gamma^\vee)$, for all $\gamma^\vee \in \Delta^{\vee+}$.

Let $\mathfrak g^\vee$ denote the Langlands dual of $\mathfrak g$. Its roots are the coroots of $\mathfrak g$. One may identify a Cartan subalgebra $\mathfrak h^\vee$ of $\mathfrak g^\vee$ with $\mathfrak h^*$ and then with $\mathfrak h$ though the Cartan scalar product.  Let $e^\vee,h^\vee,f^\vee$ be a principal s-triple for $\mathfrak g^\vee$ chosen so that $h^\vee \in \mathfrak h^\vee$. One may identify $h^\vee$ with $\rho$.

\subsection{}\label{1.2}

The Kostant (Clifford algebra) conjecture arose in the study of the Harish-Chandra map for the Clifford algebra $C(\mathfrak g)$.  It proposes that two ``naturally defined" filtrations on $\mathfrak h$ coincide. It was positively settled for $\mathfrak g$ of type $A$ in the thesis of Bazlov \cite {B1}.

The Kostant conjecture can be rephrased \cite {B2} in terms of the more familiar Harish-Chandra map $\varphi$ for $U(\mathfrak g)$. Composing $\varphi$ with evaluation at $\rho$ (which naturally arises in the study of $C(\mathfrak g)$) gives a map $\varphi_\rho:U(\mathfrak h) \rightarrow k$.  Then the first of the above filtrations obtains by applying $\varphi_\rho$ to the right hand factor in $(\mathfrak g \otimes \mathscr F^m U(\mathfrak g))^\mathfrak g$.  The second is defined by the action of $e^\vee$ on $\mathfrak h$ identified with $\mathfrak h^\vee$.

A slight generalization of Bazlov's result relevant to this latter context has been announced by Alekseev and Moreau \cite {AM}.  In this the Kostant conjecture is proved in type $A$ but with respect to evaluation at any multiple of $\rho$.

The Kostant conjecture admits a version in which the enveloping algebra is replaced by the symmetric algebra and the Harish-Chandra map by the Chevalley restriction map.  In this context it has been positively settled by Rohr \cite {R}.  However this solution is deceptively simple and so far attempts to extend this to the enveloping algebra context have failed.  The trouble is that the lower order terms introduced in the non-commutative context cannot be ignored or easily spirited away. (It would seem that the attempt to do this in \cite {B2} has a difficulty.)  Indeed the Kostant conjecture admits a generalization (at least if $\mathfrak g$ is simply-laced) in which $\mathfrak g$ appearing in the left hand factor above is replaced by any simple finite dimensional $\mathfrak g$ module $V$.  This has a positive answer in the symmetric algebra version; but fails even for the $27$ dimensional representation of $\mathfrak {sl}(3)$ in the enveloping algebra version.  

\subsection{}\label{1.3}

The Kostant conjecture in the general context of an arbitrary simple module $V$ occurring in $U(\mathfrak g)$ can be shown \cite {J2} to be equivalent to the coincidence of the ``operator" and ``degree" filtrations on the dual $\delta M(0)$ of a Verma module\footnote{We refer to this as yet unpublished note only for information.  Its content will not be used here.}. It is rather easy to appreciate that this must fail badly even though the corresponding graded objects are isomorphic. Unfortunately this failure persists even when we restrict to that part of $\delta M(0)$ coming from modules arising in $C(\mathfrak g)$, which form a much smaller class.  This does not rule out the possibility that these filtrations coincide on that part $\delta M(0)$ coming from the adjoint module. Indeed the adjoint module admits a rather canonical presentation in $C(\mathfrak g)$ through Chevalley transgression and this presentation may well have been what lay behind the Kostant conjecture.  However so far we have been unable to make use of this technology.

\subsection{}\label{1.4}

The Kostant conjecture serves no particular purpose except that being as we now know a very difficult question, it forces us to make an in-depth study of the structures which lead to its formulation.   In view of this we propose to study an ``analogue Kostant conjecture" in which the Harish-Chandra map is replaced by the ``generalized Harish-Chandra" map. Here a distinct advantage is the image of the map is known to be that part of $\mathfrak h\otimes S(\mathfrak h)$ invariant under the set $\Xi$ of Zhelobenko \cite {Z} operators.  The proof of this statement (which holds for any $V$) is the subject of a recent work of Khoroshkin, Nazarov and Vinberg \cite {KNV}. An alternative proof has been given in \cite {J1}.  The latter is based on the computation of a determinant of the transition matrix from the zero weight subspace $V_0$ of $V$ and the generators of the Zhelobenko invariants.

\subsection{}\label{1.5}

The computation of the Zhelobenko invariants seems to be a rather daunting task even when $V$ is the adjoint representation.  Whilst \cite [Sect. 4]{KNV} describe these invariants implicitly, their description relies on fixing a basis of $V_0$ adapted to each simple root in turn.  Consequently it is completely unclear what should be explicit formulae for these invariants.

In the case of the adjoint representation we fix a basis in $V_0=\mathfrak h$ given by the set $\{\varpi_i\}_{i\in I}$ of fundamental weights.  Then a Zhelobenko invariant $J$ takes the form $\sum_{i\in I} \varpi_i\otimes  q_i$, for some $q_i \in S(\mathfrak h)$.  The leading order terms of the $q_i$ take the form $\partial q/\partial \varpi_i$ where $q \in S(\mathfrak h)^{W.}$, that is to say an invariant under the translated action of the Weyl group $W$.  Even the latter can hardly be described as being known explicitly except perhaps for the classical Lie algebras. Moreover even in type $A$ the lower order terms in the Zhelobenko invariants do not follow a particularly simple pattern.

\subsection{}\label{1.6}

Our present goal is to obtain just enough information on the Zhelobenko invariants to settle the analogue Kostant conjecture. It was a surprise that this involves the Bernstein-Gelfand-Gelfand (or simply, BGG) operators which were introduced \cite {BGG} to study the cohomology of the flag variety.   In more simple-minded terms the BGG operators give the dimension polynomials for the Demazure modules starting from the product of the positive roots.  Quite coincidentally (so it would seem) this latter polynomial is up to a minor modification just the determinant associated to the Zhelobenko operators mentioned in \ref {1.4} for the adjoint module.

\subsection{}\label{1.7}

Just as in \ref {1.2} it is possible to formulate the analogue Kostant problem in the context of an arbitrary simple module occurring in $U(\mathfrak g)$.  At present its solution is beyond our present means.  Moreover whilst the analogue Kostant conjecture for the adjoint module is still very natural in the context of the Clifford algebra and may yet prove to be equivalent to the original Kostant conjecture is this narrower framework, there is much less reason for the latter to be true in general.

\

\textbf{Acknowledgements.}  This work had two inspirations.  First, Alekseev persuaded me to work on the Kostant conjecture during my stay in Geneva in June 2010.  Some results were obtained but there were rather inconclusive \cite {J2}. Alekseev seemed to think one might be able to make use of the Zhelobenko operators.  This did not settle the Kostant conjecture but led to the present work.  Second, Nazarov introduced us to the Zhelobenko operators during his visit to the Weizmann Institute during March 2010 and in particular to his result with Khoroshkin and Vinberg \cite {KNV}.  In May 2011, I was his guest in York during which time I was able to give an alternative proof \cite {J1} of their theorem.  Finally I tried to compute the Zhelobenko invariants for the adjoint case which led to the present work.  The results obtained here for the simply-laced case, which is rather easier, were presented at our seminar in the Weizmann Institute in August 2011.

I would particularly like to thank Anton Alekseev and Maxim Nazarov for their hospitality and inspiration.  I should also like to acknowledge my former student Yuri Bazlov who settled the Kostant conjecture for type $A$ in his phD thesis \cite {B1}  presented at the Weizmann Institute in 2000.

\section{The Basic Identity}

\subsection{}\label{2.1}

Let
$$\mathfrak g = \mathfrak n^- \oplus\mathfrak h \oplus
\mathfrak n^+,$$ be the triangular decomposition of $\mathfrak g$ coorresponding to our choices in \ref {1.1}.

Let $V$ be a simple finite dimensional $U(\mathfrak g)$ module (eventually just the adjoint module). Define a right action of $U(\mathfrak g)$ on $V\otimes U(\mathfrak g)$ by right multiplication on the second factor and a left action given by
$$x(v\otimes a)=xv \otimes a +v\otimes xa, \forall x \in \mathfrak
g, a \in U(\mathfrak g), v \in V.$$

 Define an action of $\ad x :x \in \mathfrak g$ on $V \otimes U(\mathfrak g)$ by $\ad x (v \otimes a) :=x(v \otimes a)+ (v \otimes a)x = xv \otimes a + v \otimes [x,a]$.

Starting from the lowest weight vector for $V$ an easy induction
argument shows that we have a direct sum decomposition $V\otimes U(\mathfrak g) =V \otimes U(\mathfrak h) \oplus (\mathfrak n^-(V\otimes U(\mathfrak g))+(V\otimes U(\mathfrak g))\mathfrak n^+)$.  Let $\hat{\Phi}$ denote the projection onto the second factor. It is called the generalized Harish-Chandra map.

It follows from the above that
$$\hat{\Phi}((V\otimes U(\mathfrak g))^\mathfrak g) \subset V_0 \otimes S(\mathfrak h).$$

Zhelobenko \cite {Z} has defined a set $\Xi:=\{\xi_i\}_{i\in I}$ of operators on (a localization of) $V\otimes U(\mathfrak g)$ which act by the identity on invariants and with the marvelous property that they pass through the generalized Harish-Chandra map to operators (which we shall denote by the same symbols) on (a localization of) $V \otimes U(\mathfrak h)$.  A basic result proved by Khoroshkin, Nazarov and Vinberg \cite {KNV} is that $\hat{\Phi}$ induces an isomorphism of $(V\otimes U(\mathfrak g))^\mathfrak g$ onto $(V_0 \otimes S(\mathfrak h))^\Xi$.
We refer to it as the KNV theorem.

\subsection{}\label{2.2}

The Zhelobenko operators and their main properties are reviewed in \cite {KNV} and in \cite {J1}.  For our own convenience we shall use the latter as a reference.

Recall that there is a translated Weyl group action on $\mathfrak h^*$ defined by $w.\lambda=w(\lambda +\rho)-\rho$.  This induces a translated Weyl group action on $S(\mathfrak h)$ by identifying the latter with the algebra of polynomial functions on $\mathfrak h^*$ and transport of structure.  This immediately implies that
$$w.qp=(w.q)(w.p), \forall w\in W, q,p \in S(\mathfrak h), \eqno {(1)}$$

Again let $h_i=\alpha_i^\vee$ be the coroot corresponding to $\alpha_i$. Then for all  $h\in \mathfrak h, m \in \mathbb Z$ one  $s_i.(h+m)(\lambda)=(h+m)(s_i.\lambda)=h(\lambda - (h_i(\lambda) +1)\alpha_i)+m=h(\lambda)+m -h(\alpha_i)(h_i(\lambda)+1)$, that is
$$s_i.(h+m)= h+m-h(\alpha_i)(h_i+1). \eqno {(2)}$$

In addition Zhelobenko \cite [Lemma 2.3]{J1} showed the

\begin {lemma}  For all $i \in I$, $a \in V \otimes S(\mathfrak h)$, $b \in S(\mathfrak h)$ one has $\xi_i(ba)=(s_i.b)\xi_i(a)$.
\end {lemma}

\subsection{}\label{2.3}

It follows from the KNV theorem and the previous lemma that $(V_0 \otimes S(\mathfrak h))^\Xi$ is a free $S(\mathfrak h)^{W.}$ module on $\dim V_0$ generators. Their leading order terms (cf \cite [Sect. 5]{KNV}) are just the set of generators for the free $S(\mathfrak h)^W$ module obtained as the image of the Chevalley restriction map of $(V\otimes S(\mathfrak g))^\mathfrak g$ into $(V_0\otimes S(\mathfrak h))^W$ (which is \textit{not} generally surjective). In the case when $V$ is the adjoint module we may label these generators by $I$, specifically as $\{J_i\}_{i\in I}$.

\subsection{}\label{2.4}

From now we just take $V$ to be the adjoint module.  In this case the Chevalley restriction map \textit{is} surjective.  Let $J$ be a Zhelobenko invariant.  As in \ref {1.5} we write
$$J=\sum_{i\in I} \varpi_i \otimes q_i,$$
for some $q_i \in S(\mathfrak h)$.

Our aim is to compute the $q_i$ as far as is needed to settled the analogue Kostant conjecture. The following is a first reduction.

The leading order terms of the $q_i$ take the form $\partial q/\partial \varpi_i$ for some $W$ invariant $q$.  Moreover the latter is divisible by the co-root $h_i$.  Our first basic result is the following

\begin {prop} One has $p_i:=q_i/(h_i+2) \in S(\mathfrak h)$ and
$$(p_j-s_i.p_j)(h_j+2)=h_j(\alpha_i)(h_i+1)(p_i-s_i.p_j). \eqno {(3)}$$
\end {prop}

\begin {proof} We may write $$\alpha_i=2\varpi_i-\alpha_i^\perp,$$
with $\alpha_i^\perp :=-\sum_{j\in I \setminus \{i\}}h_j(\alpha_i)\varpi_j$, being orthogonal to $\alpha_i$.

Let $e_i,h_i,f_i$ be the s-triple defined by $i \in I$ and recall \cite [Eq. (1) and 2.3] {J1} that $\xi_i=\eta_is_i$, where $\eta_i(h \otimes 1)= (\ad e_i \ad f_i) h \otimes h^{-1}_i$.
Substituting from the above we obtain
$$2\xi_i(\varpi_i\otimes 1)=\eta_i(\alpha_i^\perp-\alpha_i)\otimes 1=(\alpha_i^\perp-\alpha_i)\otimes 1 -2\alpha_i\otimes h_i^{-1}=(\alpha_i^\perp\otimes 1)-\alpha_i\otimes \frac{h_i+2}{h_i}.$$

In view of Lemma \ref {2.2} we obtain
$$\xi_i(\varpi_i\otimes q_i)=\frac{1}{2}(\alpha^\perp_i\otimes s_i.q_i)-(\varpi_i-\frac {\alpha_i^\perp}{2})\otimes \frac {h_i+2}{h_i}s_i.q_i.\eqno {(4)}$$

On the other hand by Lemma \ref {2.2} again we have
$$\sum_{j\in I\setminus \{i\}}\xi_i(\varpi_j\otimes q_j)= \sum_{j\in I\setminus \{i\}}\varpi_j\otimes s_i.q_j. \eqno {(5)}$$

In the sum of the right hand sides of $(4)$ and $(5)$ there are no terms proportional to $\varpi_i$ and so the invariance of the sum of the left hand sides, which is $J$, implies that $q_i=-\frac {h_i+2}{h_i}s_i.q_i$, that is $q_i$ is divisible by $h_i+2$.  Moreover by $(1)$ and $(2)$ the quotient $p_i$ satisfies $s_i.p_i=p_i$.

Again by the invariance of $J$ and equating the coefficients of $\varpi_j$ we obtain from $(4)$ and $(5)$ that
$$q_j-s_i.q_j=-\frac {1}{2}h_j(\alpha_i)(s_i.q_i-q_i)=h_j(\alpha_i)(h_i+1)p_i.\eqno {(6)}$$

On the other hand by $(1)$ and $(2)$, we obtain $$s_i.q_j=s_i.(h_j+2)s_i.p_j=((h_j+2)-h_j(\alpha_i)(h_i+1))s_i.p_j, $$ which on substitution into $(6)$ gives  $(3)$.
\end {proof}

\subsection{}\label{2.5}

It is clearly inconvenient to carry the burden of the translated Weyl group action.  Therefore as in \cite [3.1]{J1} we use the automorphism $\theta$ of $S(\mathfrak h)$ defined by $\theta(q)(\lambda)=q(\lambda + \rho)$, which has the property that $w.\theta(q)=\theta(wq)$.  Observe further that $\theta(h_i+m)=h_i+m+1$.  Now define new polynomials $P_i:=\theta^{-1}(p_i)$.  Substitution in $(3)$ gives
$$(h_j+1)(P_j-s_iP_j)=h_j(\alpha_i)h_i(P_i-s_iP_j). \eqno {(7)}$$

Following \cite {BGG} we introduce linear operators $A_i:i\in I$ on $S(\mathfrak h)$ by the formulae
$$A_if:=\frac{f-s_if}{h_i}, \forall f \in S(\mathfrak h). \eqno {(8)}$$

This gives the

\begin {cor}  For all $i,j \in I$ one has
$$A_iP_j=h_j(\alpha_i)\frac{P_i-P_j}{1+s_i(h_j)}. \eqno {(9)}$$
\end {cor}

\begin {proof}  Subtract $h_j(\alpha_i)h_i(P_j-s_iP_j)$ from both sides of $(7)$. Since $h_j+1-h_j(\alpha_i)h_i=1+s_i(h_j)$, the assertion results.
\end {proof}

\subsection{}\label{2.6}

Equation $(9)$ is what we must solve in order to determine the Zhelobenko invariant $J$ in the adjoint case. Of course this is not too easy as there are infinitely many solutions.  The simplest solution described in \cite [3.6]{J1} is when all the $P_j$ equal $1$.

\section{A Reduction}

\subsection{}\label{3.1}

Our goal is just to use Corollary \ref {2.5} to prove the analogue Kostant conjecture.  We now describe what must be done to achieve this.

\subsection{}\label{3.2}

The details given in this subsection were explained to me by Alekseev.  More details and further considerations can be found in the thesis of his student (R. H. Rohr) published (in part) in \cite {R}.

Identify $\mathfrak h^\vee$ with $\mathfrak h$ as in \ref {1.1}.  Fix $q \in Y(\mathfrak g^\vee)$.  Its differential $dq$ evaluated at a multiple $s\rho$ of $\rho$ can be written in the form
$$dq(s \rho)=\sum_{i \in I} \varpi_i  \partial q/\partial \varpi_i(s \rho).$$

The invariance of $q$ under $\exp {te^\vee}: t \in k^*$ (viewed as an indeterminate) implies that
$$ \exp te^\vee(\sum_{i \in I} \varpi_i  \partial q/\partial \varpi_i(s \rho))=\sum_{i \in I} \varpi_i  \partial q/\partial \varpi_i(\exp te^\vee (s \rho)). \eqno {(10)}$$
Yet since $\rho$ is just the semisimple element $h^\vee$ of the principal s-triple for $\mathfrak g^\vee$, one has $\exp te^\vee (s \rho)= s\rho-2ste^\vee$, which is linear in $t$.  Thus if $q$ is a polynomial of degree $m+1$, it follows the right hand side of $(10)$ is a polynomial in $t$ of degree at most $m$.  Thus
$$ (e^\vee)^{m+1}(\sum_{i \in I} \varpi_i  \partial q/\partial \varpi_i(s \rho))=0. \eqno {(11)}$$

 Now this last identity only depends on the image of $q$ under the Chevalley restriction map.  Identifying $\mathfrak h^\vee$ with $\mathfrak h$ as in \ref {1.1} the  Chevalley isomorphisms applied to the two invariant algebras $Y(\mathfrak g^\vee), Y(\mathfrak g)$ have the \textit{same} image, namely $S(\mathfrak h)^W$.  Thus we can equally well view $q$ as an element of $Y(\mathfrak g)$.  Thus we obtain the following

 \begin {prop} Define an action of $e^\vee$ by identifying $\mathfrak h$ with the Cartan subalgebra of the Langlands dual algebra $\mathfrak g^\vee$.  Then any invariant polynomial $q$ on $\mathfrak g^*$ of degree $m+1$ satisfies $(11)$.
\end {prop}

\subsection{}\label{3.3}

The above result (due to Alekseev-Rohr) can be viewed as settling the Kostant conjecture for the symmetric algebra.   Here one can ask if it can be extended to the case when the adjoint module is replaced by an arbitrary finite dimensional simple $\mathfrak g$ module $V$.  In this case the term in $(11)$ lying in parentheses must be replaced by an element of $(V_0\otimes \mathscr F^m(S(\mathfrak g))^W$. If $\mathfrak g$ is simply-laced the proof follows exactly the same lines as in the case of the adjoint module.  The general question was considered in \cite [Sect. 2]{J2}.  The point is that one must embed the zero weight space $V_0$ as a $W$ module in the zero weight space of some finite dimensional  $\mathfrak g^\vee$ module $\textbf{V}$ in order to extend the action of $e^\vee$.  This can be achieved by a rather general universality argument and from which the required version of $(11)$ results.  However in this it is not assured that $m+1$ is the \textit{smallest} integer satisfying the required version of $(11)$.  However if $V_0$ is a \textit{simple} $W$ module (which is a rather restrictive condition) then one may also assume that $\textbf{V}$ is a simple $\mathfrak g^\vee$ module and then indeed $m+1$ can be shown \cite [Sect. 2]{J2} to be the smallest integer satisfying the required version of $(11)$.

Unfortunately in the above more general context the Kostant conjecture has a negative answer.  Indeed Alekseev had already shown to me a computation which implies that the Kostant conjecture has a negative answer for the simple $27$ dimensional module $V$ in $\mathfrak {sl}(3)$.  This occurs with multiplicity three, more precisely once in degrees $2,3,4$, in the harmonic subspace of $S(\mathfrak g)$.  Yet $4$ is the smallest value of $m$ for which $(e^\vee)^{m+1}(V_0\otimes \varphi_\rho(\mathscr F^2(U(\mathfrak g))=0$. This example is particularly inopportune because $V$ occurs in $C(\mathfrak g)$, viewed as a $\mathfrak g$ module via the Chevalley-Kostant construction \cite {K2}, which is the context of the original Kostant conjecture. Actually $V$ occurs with multiplicity one in the appropriate ``harmonic" subspace of $C(\mathfrak g)$ (which by \cite [Prop. 20] {K2} we may identify with $\End V(\rho)$) so it is as if we have to ``forget" the first two values of $m$.  In any case it shows that the Kostant conjecture goes deeper than just ``formal" analysis - for example in the sense of \cite [Sect. 3]{J2}.

\subsection{}\label{3.4}

Giving a positive answer to the Kostant (or analogue Kostant) conjecture means that we require a similar assertion to $(2)$ when $\partial q/\partial \varpi_i$ is replaced by the corresponding image of the isotypical component of $\mathscr F^mU(\mathfrak g)$ of type $\mathfrak g$.  The ``easy" way to show this by proving that the resulting element of $\mathfrak h$ is proportional to $\sum_{i \in I}\varpi_i  \partial q/\partial \varpi_i(s\rho)$.

Now although top order terms do satisfy the above condition,  this is by no means obvious for the remaining terms.  Indeed for the Kostant (or analogue Kostant) conjecture that there the choice of multiples of $\rho$ is crucial not just to obtain Proposition \ref {3.2} but also to ensure that the lower order terms behave in the desired fashion.  Added to this we found that even for the lowest order terms it is not possible to take smaller values of $m+1$ to obtain the desired vanishing.

\subsection{}\label{3.5}

Recall \ref {2.5} and set $m-1 = \max \deg P_i:i \in I$.  Let $P_i^0$ denote the (leading order) term of $P_i$ of degree $m-1$.  Not all of them can be zero.  We show below that they are all non-zero.

It is immediate from $(9)$ that
$$A_iP^0_j=h_j(\alpha_i)\frac{P^0_i-P^0_j}{s_i(h_j)}, \forall i,j \in I. \eqno {(12)}$$

Suppose there are some $P^0_i$ which are zero.  Since the Dynkin diagram is connected we may assume that $i,j$ are neighbours with $P^0_i=0$ and $P^0_j\neq 0$.  Then $(12)$ gives $s_i(h_j)A_iP_j^0=-h_j(\alpha_i)P_j^0$. Now $A_i(s_i(h_j))=-h_j(\alpha_i)$, whilst $A_iP_j^0$ is $s_i$ invariant.  Then applying $A_i$ to both sides of this last equation and cancelling the non-zero scalar gives $A_iP_j^0=P_j^0$, which contradicts the fact that $A_i$  is of square zero and the choice of $P_j^0$.

Here one may remark that if $p_i^0$ (resp. $q_i^0$) denotes the leading term of $p_i$ (resp. $q_i$) as defined in \ref {2.4}, then $p_i^0=P_i^0$ and $q^0_i=p_i^0h_i$.  It follows from general considerations (as in say \cite [Sect. 5]{KNV}) or by simply repeating the analysis in \ref {2.4} and \ref {2.5}, that the divided differentials
$h_i^{-1}\partial q/\partial \varpi_i$, for $q \in S(\mathfrak h)^W$ satisfy $(12)$ and moreover give its most general solution (either by \cite [Sect. 5]{KNV} again or by reversing the argument in \ref {2.5}, that is to say by showing that if $P^0_i$ is a solution to $(12)$, then $\sum_{i\in I}\varpi_i\otimes h_iP^0_i$ is $W$ invariant).

Now observe the easy (but crucial !) fact that the $h_i$ all take the constant value $s$ on $s\rho$. Then one obtains from Proposition \ref {3.2} the

\begin {cor}  $(e^\vee)^{m+1}(\sum_{i \in I}\varpi_iP_i^0(s\rho)=0), \forall s \in k.$
\end {cor}

\subsection{}\label{3.6}

To establish the truth of the analogue Kostant conjecture it is enough via Corollary \ref {3.5} to show that the vectors of $\ell$-tuples $(P^0_i(s\rho))_{i\in I}$ and $(P_i(s\rho))_{i\in I}$ are proportional.  This can involve some choices since the lower order terms in the $P_i$ may include those coming from Zhelobenko invariants of lower degree.  We shall use this flexibility to avoid having to completely describe the solutions of $(9)$.  Nevertheless some information (specifically Proposition \ref {7.8}) on its solutions is required and this is where the BGG operators will play an important role.  The information gleaned from this result together with the fact that the $h_i$ all take the constant value $s$ on $s\rho$ will complete our proof.

\section{The BGG Operators and the BGG Monoid}

In this section we review some well-known properties of the BGG operators.  Set $K= \text {Fract} \ S(\mathfrak h)$.

\subsection{}\label{B.1}

We have already noted that $A_i$ has square zero.  Actually this can be put into a more general context by noting that $s_iA_i=A_i$ and that $A_if=0 \Leftrightarrow s_if=f$.

The $A_i:i\in I$ satisfy the braid relations. One can easily check this ``by hand". Thus if $w=s_{i_1}s_{i_2}\ldots s_{i_t}$ is a reduced decomposition then $A_w:=A_{i_1}A_{i_2}\ldots A_{i_t}$ is independent of the reduced decomposition chosen. One calls $\ell(w)=t$ the reduced length of $w$.  This latter (equivalent) fact is proved (again purely combinatorially) in \cite [Thm. 3.4]{BGG}.

Let $y\leq w$ be the Bruhat order on $W$. Then $A_w$ is a linear combination of the $y\in W|y\leq w$ with coefficients in $K$.  The non-vanishing of the coefficient of $w$ is enough to imply that the $A_w:w \in W$ are linearly independent over $k$ and therefore span an algebra of dimension $|W|$ defined by generators (that is the $A_i:i \in I$) and relations (that is the vanishing of squares and the braid relations - alternatively if one prefers, the relations in \cite [Thm. 3.4]{BGG}).  We call the monoid $\textbf{A}$ generated by the $A_i:i \in I$ satisfying the above relations, the BGG monoid.  Given $A \in \textbf{A}$ observe that the length of $A$ viewed as a word formed from the letters $A_i:i \in I$ is \textit{in}dependent of presentation.  We denote it by $\ell(A)$.  One has $\ell(A_w)=\ell(w)$, for all $w \in W$.

The BGG operators are limits of the Demazure operators which also satisfy the braid relations \cite {De}; but are idempotent. For a similar reason to the above they are linearly independent. Hence they also form an algebra (often referred to as the singular Hecke algebra) defined by generators and relations.  The Demazure operators give the characters of the Demazure modules.  The original proof had a (serious) error but several correct proofs were given in what is now a long story. It gives an ``abstract" proof that Demazure operators braid and hence so do the $A_i:i \in I$.  The limits of these characters give the ``BGG dimension polynomials", describing the dimensions of the Demazure modules.  These dimension polynomials may also be obtained by the action of the $A_w: w \in W$ on the product of the roots (which is the Weyl dimension polynomial describing the dimensions of the simple finite dimensional $U(\mathfrak g)$ modules).  This result is the subject of \cite {BGG} when the latter is expressed in simple-minded terms.

\subsection{}\label{B.2}

The BGG operator $A_i$ acts like a skew derivation.  That is we have

$$A_i(fg)=s_i(f)A_i(g)+fA_i(g)=fA_i(g)+s_i(f)A_i(g). \eqno {(13)}$$

The $A_i$ \textit{do not} preserve $W$ invariant subspaces (except when $\mathfrak g$ has rank one).  This is compensated by the following observation.

Let $L$ be the homogeneous ideal of $S(\mathfrak h)$ generated by the augmentation ideal of $S(\mathfrak h)^W$.  It is clear from $(13)$ that $A_iL\subset L$, for all $i \in I$.  Consequently the action of the $A_i:i \in I$ on $S(\mathfrak h)$ passes to the quotient $Q:=S(\mathfrak h)/L$.

\begin {lemma} Suppose $f\in Q$ satisfies $A_if=0, \forall i \in I$.  Then $f$ is a scalar.
\end {lemma}

\begin {proof} As noted in \ref {B.1}, the hypothesis is equivalent to $s_if=f, \forall i \in I$, that is to say to $f$ being $W$ invariant.   Hence the assertion.
\end {proof}

\subsection{}\label{B.3}

The natural gradation on $S(\mathfrak h)$ descends to $Q$.  The following is an implicit consequence of \cite [Thm. 3.14]{BGG} (since the BGG dimension polynomials form a basis of $Q$). We give an easy proof.

\begin {cor}   Suppose $f \in Q$ has degree $m$.   Then there exists $w \in W$ of length $m$ such that $A_wf$ is a non-zero scalar.  Moreover if $A_if\neq 0$, then we may assume that $\ell(ws_i) = \ell(w)-1$.
\end {cor}

\begin {proof} We can assume $f$ homogeneous.  Then $A_if$ is homogeneous of degree $m-1$ or zero.  Let $n \in \mathbb N$ be maximal such that $A_yf \neq 0$, for some $y \in W$ of length n.  Then $n \leq m$.  If a strict inequality held then we would obtain a homogeneous element $g \in Q$ of degree $m-n$ which is non-zero and yet annihilated by all the $A_j:j\in I$.  This contradicts the conclusion of Lemma \ref {B.2}.
\end {proof}

\section{Exponents}

\subsection{}\label{C.1}

Recall the result of Chevalley that $S(\mathfrak h)^W$ is a polynomial algebra on $\ell$ generators which can be assumed to be homogeneous.   The degrees $m_i+1:i\in I$ of these generators can assumed to be increasing.  The $m_i:i\in I$ are called the \textit{exponents} of $\mathfrak g$. They are same for $\mathfrak g^\vee$. One has $m_1=1$ and $m_\ell=\rho(\beta^\vee_0)$, where $\beta_0$ is the unique highest root.  After Kostant \cite {K} the dimensions of the simple submodules of $\mathfrak g$ under the action of a principal s-triple are the $2m_i+1:i \in I$.  In particular $(e^\vee)^{m_\ell+1}(\mathfrak h)=0$, in the sense of \ref {3.2}.  On the other hand as we have already seen $(e^\vee)^{m_1+1} s\rho =0, \forall s \in k$.

\subsection{}\label{C.2}

Recall \ref {3.5} and let $m-1$ be the common degree of a set $\{P_i:i \in I\}$ of solutions to $(9)$ coming from a Zhelobenko invariant $J$.  To settle the analogue Kostant conjecture it is obviously enough to take $J$ to be one of the free generators of $(\mathfrak h \otimes S(\mathfrak h)))^\Xi$.  Then the leading term of $J$ is given by differential of an invariant generator. This means in particular that we may assume $m$ to be an exponent.  Then we must show that
$$(e^\vee)^{m+1}\sum_{i\in I}\varpi_iP_i(s\rho)=0, \forall s \in k.$$

By the remarks in \ref {2.6} and \ref {C.1}, this holds trivially if either $m=m_1$, or $m=m_\ell$. In particular the analogue Kostant conjecture holds trivially in rank $2$ (as was well-known for the Kostant conjecture itself).

\section{The Zhelobenko Monoid}

\subsection{}\label{6.1}

Set $m_{i,j}=h_i(\alpha_j)h_j(\alpha_i):i,j \in I$.  Recall that $m_{i,j}\in \{0,1,2,3,4\}$.

  View the $P_i: i \in I$ as polynomials on $\mathfrak h^*$ satisfying $(14)$.  The action on the $A_i: i \in I$ on these elements gives a finite set $\textbf{P}$ of polynomials.  A straightforward calculation gives

$$\begin {array} {rccl} A_jP_i&=&0 &:m_{i,j}=0,\\\\
A_jP_i&=&\frac {P_i-P_j}{(1+\alpha_i^\vee+\alpha_j^\vee)}&:m_{i,j}=1, \\\\
A_j A_i P_j&=&\frac{2(P_i-P_j)}{(1+\alpha_i^\vee+\alpha_j^\vee)(1+\alpha_i^\vee+2
\alpha_j^\vee)} &:m_{i,j}=2,\\\\
 A_j A_i A_j A_i P_j& =&\frac{6(P_j-P_i)}{(1+\alpha_i^\vee+\alpha_j^\vee)(1+\alpha_i^\vee+2\alpha_j^\vee)
(1+\alpha_i^\vee+3\alpha_j^\vee)(1+2\alpha_i^\vee+3\alpha_j^\vee)}&:m_{i,j}=3,\\

\end {array} \eqno{(14)}$$
where $\alpha^\vee_j$ is assumed to be the shorter of the two coroots.

\

It follows from $(14)$ that the $A_i,P_j :i,j \in I$ satisfy the following relations

$$\begin {array}{rccrcll} A_i^2&=&0,&A_i P_i& =&0,\\
 A_i A_j& =& A_j A_i, & A_i P_j& =&0 &:m_{i,j}=0, \\
 A_iA_jA_i&=&A_jA_iA_j, &A_iP_j&=&-A_jP_i&:m_{i,j}=1,\\
 (A_iA_j)^2&=&(A_jA_i)^2, &A_jA_iP_j&=&-A_iA_jP_i&:m_{i,j}=2,\\
 (A_iA_j)^3&=&(A_jA_i)^3, &(A_jA_i)^2P_j&=&-(A_iA_j)^2P_i&:m_{i,j}=3.\\

 \end {array}\eqno{(15)}$$

  We call the pair $(\textbf{A},\textbf{P})$ satisfying the above relations the Zhelobenko monoid.  Notice that as a word the length $\ell(P)$ of $P \in \textbf{P}$ is independent of presentation.  Again the subset of $I$ of letters occurring in $P \in \textbf{P}$ is independent of presentation.  It is denoted by Supp $P$.

 \subsection{}\label{6.2}

 The structure of the Zhelobenko monoid is made more complicated by the sign changes forced by the relations in $(15)$.  This problem is analogous to the problem of how to choose signs in the elements $x_\alpha:\alpha \in \Delta$ of a Chevalley basis.  (The latter problem was solved by Tits \cite {T}.)  In Section $7$ we shall use the \textit{existence} of $\mathfrak g$ to solve this sign problem. In the remainder of Section $6$ we just identify two elements of $\textbf{P}$ if they differ by a change of sign.

 The structure of the Zhelobenko monoid is significantly simpler in the simply-laced case.  Thus in this and the next two subsections we shall assume that $\mathfrak g$ is simply-laced.

 \begin {lemma} ($\mathfrak g$ simply-laced) Suppose $A_{i_1}A_{i_2}\ldots A_{i_r}P_{i_{r+1}}\neq 0$ and set $\beta_t=\sum_{s=t}^{r+1} \alpha_{i_s}$.  Then $\alpha^\vee_{i_{t-1}}(\beta_t)<0, \forall t =2,3,\ldots,r+1$.  In particular $\beta_t$ is a positive root for all $t=1,2,\ldots,r+1$.
 \end {lemma}

 \begin {proof}  The proof is by induction on $r$.  If $r=1$ the assertion follows from the first and second lines of $(15)$.

  Consider the case $\alpha^\vee_{i_1}(\beta_2)>0$.

  In this case $\alpha^\vee_{i_1}(\alpha_{i_2}+\beta_3)>0$.  By the first line of $(15)$ we may assume $i_1 \neq i_2$.  Now $\beta_3$ is a positive root by the induction hypothesis and then since $\mathfrak g$ is simply-laced the above inequality forces $\alpha^\vee_{i_1}(\alpha_{i_2})=0$ and $\alpha^\vee_{i_1}(\beta_3)>0$. Then the required assertion follows from the second line of $(15)$ and the induction hypothesis.

   Consider the case $\alpha^\vee_{i_1}(\beta_2)=0$.

   In this case $\alpha^\vee_{i_1}(\alpha_{i_2}+\beta_3)=0$ and by the first line of $(15)$ we may assume $i_1 \neq i_2$.

   Suppose that $\alpha^\vee_{i_1}(\alpha_{i_2})=0$.  Then $\alpha^\vee_{i_1}(\beta_3)=0$. Then the required assertion follows from the second line of $(15)$ and the induction hypothesis.

    Next suppose that $\alpha^\vee_{i_1}(\alpha_{i_2})=-1$ and $\alpha^\vee_{i_1}(\beta_3)=1$.

    Admit that $i_1=i_3$.  Then $\alpha^\vee_{i_1}(\beta_4)=-1$ and so $\alpha_{i_2}+\beta_4$ cannot be a root (since $\mathfrak g$ is simply-laced) and so $A_{i_2}A_{i_4}\ldots P_{i_{r+1}}=0$, by the induction hypothesis.  Using the left-hand side of $(19)$ gives the assertion in this case.

    Finally admit that $i_1 \neq i_3$.  Then since $\mathfrak g$ is simply-laced we obtain $\alpha^\vee_{i_1}(\alpha_{i_3})=0$ and $\alpha^\vee_{i_1}(\beta_4)=1$.  Then through the left hand-side of the second line in $(19)$ we obtain the required result by repeating the previous argument.

 \end {proof}

 \subsection{}\label{6.3}

  \begin {lemma} ($\mathfrak g$ simply-laced)  Suppose $$A_{i_1}A_{i_2}\ldots A_{i_r}P_{i_{r+1}}=\pm A_{j_1}A_{j_2}\ldots A_{j_s}P_{j_{s+1}},$$
 is non-zero.  Then $r=s$ and up to a change of sign the right-hand side may be re-expressed through lines one to three in $(15)$ such that $i_t=j_t, \forall t =1,2,\ldots,r+1$.
  \end {lemma}

  \begin {proof}

    Apply Lemma \ref {6.2} to the non-vanishing of the left-hand side.  In the notation of this lemma it follows that $\beta_1$ is a positive root satisfying $o(\beta_1)=r+1$.  A similar assertion holds for the right-hand side forcing $r=s$ and $\beta_1=\sum_{t=1}^{s+1} \alpha_{j_t}$. Moreover from this last expression $\alpha_{j_1}^\vee(\beta_1) >0$, since $\mathfrak g$ is simply-laced.  We conclude the proof by induction on $r$.  In this we may assume that $i_1 \neq j_1$ for otherwise we may cancel off $A_{j_1}$.

    From the relation $\alpha_{j_1}^\vee(\alpha_{i_1}+\beta_2) >0$, and since $\mathfrak g$ is simply-laced we conclude that  $\alpha_{j_1}^\vee(\alpha_{i_1})=0$ or $\alpha_{j_1}=\beta_2$, the latter being possible only if $r=1$.  In the former case $\alpha_{j_1}^\vee(\beta_2)>0$ and we can repeat this argument until the first $t \leq r+1$ is reached with $j_1=i_t$.  In this $\alpha_{j_1}^\vee(\alpha_{i_u}):u<t$ and then using the relations in the second and third lines of $(19)$ we may cancel off $A_{j_1}$ as before.
\end {proof}

\textbf{Remark}.  Notice that all the relations in $(15)$ relevant to the simply-laced case have been used in the combined proofs of these last two lemmas.

\subsection{}\label{6.4}

Continue to assume that $\mathfrak g$ is simply-laced.

It follows from Lemmas \ref {6.2}, \ref {6.3} that there is a natural bijection $\mathscr P : \Delta^+ \iso \textbf{P}$
described by setting $\mathscr P(s_{j_1}s_{j_2}\ldots s_{j_s}\alpha_{j_{s+1}})
=\pm A_{j_1}A_{j_2}\ldots A_{j_s}P_{j_{s+1}}$, when the right hand side is non-zero.  The following is immediate.

\begin {lemma} ($\mathfrak g$ simply-laced.)

\

(i) \ $\ell(\mathscr P(\gamma))=\ell(\gamma)$, for all $\gamma \in \Delta^+$.

\

(ii) \ If  $A_\alpha \mathscr P(\gamma)\neq 0$ then  $A_\alpha \mathscr P(\gamma)=\mathscr P(s_\alpha\gamma)$, for all $\alpha \in \pi, \gamma \in \Delta^+$.

\

(iii) \  $\textbf{P}$ admits a unique element of maximal length (up to signs).
\end {lemma}

\subsection{}\label{6.5}

Drop the assumption that $\mathfrak g$ is simply-laced. Then in general the Zhelobenko monoid (even after sign identification) has more elements than $\Delta^+$.   Nevertheless rather surprisingly there is still a unique (up to signs) element in $\textbf{P}$ of maximal length.

As equation $(14)$ might suggest it is more natural to associate $A_i,P_i$ with the coroot $\alpha_i^\vee$.  Below we construct an injection of the set $\Delta^{\vee +}$ of positive coroots to $\textbf{P}$ (again ignoring signs).  Some of this construction is case by case.  We start with a general fact.

\begin {lemma}  $\textbf{P}$ possesses an element of length $m_\ell$.

\end {lemma}

\begin {proof}

  Let $q$ be a generator of $S(\mathfrak h)^W$ homogeneous of degree $m_\ell +1$. The $\partial q/\partial \varpi_i:i \in I$ define non-zero elements of $Q$ which we shall denote by the same symbols.  As in \ref {3.5}, set $P_i^0=A_i\partial q/\partial \varpi_i$, for all $i \in I$.  Recall that these elements are all non-zero (in $Q$) and are homogeneous of degree $m_\ell -1$.  By Corollary \ref {B.3} there exist $w \in W, i \in I$ such that $A_w \partial q/\partial \varpi_i$ is a \textit{non-zero} scalar.  Necessarily $\ell(w)=m_\ell$. On the other hand the $A_y \partial q/\partial \varpi_i: y\in W, i \in I$ satisfy the relations of the Zhelobenko monoid (and possibly further relations).  Hence the required assertion.

\end {proof}

\subsection{Types $B_n,C_n$}\label{6.6}

\

It is convenient first to simplify the notation for elements in $\textbf{P}$ valid without restriction on type, namely we set
$$A_{i_1}A_{i_2}\ldots A_{i_r}P_{i_{r+1}}=i_1i_2\ldots i_r (i_{r+1}).$$

In types $A_n,B_n,C_n$ we shall make the further abbreviations
$$[i,j]=i \ i-1 \ldots (j), \quad  [i,n,j]:=i \ i+1 \ldots n-1 \ n \ n-1 \ldots j+1 \ (j). $$

It is clear that $\ell([i,j])=i-j+1, \ell([i,n,j])=2n-i-j+1$, assuming that the elements in question are non-zero.

Let $\textbf{P}^A_n$ denote the Zhelobenko monoid in type $A_n$. One may easily deduce from Lemmas \ref {6.2}, \ref {6.3} the following ``canonical" form for elements of $\textbf{P}^A_n$.

 $$\textbf{P}^A_n= \{[i,j]\}_{n\geq i \geq j \geq 1}.$$

 \

In types $B_n,C_n$, we use the Bourbaki convention that $\alpha_n$ is the simple root which has a distinct length to the remaining simple roots.  Then the first $n-1$ simple roots form a subsystem of type $A_{n-1}$ in $B_n,C_n$.

 Let us use $\textbf{P}^{BC}_n$ to denote the Zhelobenko monoid in types $B_n,C_n$, which is of course the same in both cases.  We describe a canonical form for the elements of $\textbf{P}^{BC}_n$.

 One has the following relations

 $$\begin {array}{rcl}
 A_n[n-1,j]&=&[n,n,j], \\
 A_n[i,n,n]&=&[i,n,n-1]:i<n, \\
 A_{i-1}[i,n,j]&=&[i-1,n,j], \\
 A_{j-1}[i,n,j]&=&[i,n,j-1]:j<i,\\
 A_j[i,n,j]&=&[i,n,j-1]:n>j>i.\\
 \end{array} \eqno{(16)}$$

 All remaining expressions (\textit{not involving $n$}) are zero.  Those not involving $n$ are determined by \ref {6.4}.

 \begin {lemma}

  \

  \

  (i)  $\textbf{P}^{BC}_n \setminus \textbf{P}^A_{n-1}= \{[i,n,j]\}_{i,j =1}^n$.

\

(ii)  Every element of $\textbf{P}^{BC}_n$ has length $\leq 2n-1$ and $[1,n,1]$ is the unique element of length $m_\ell=2n-1$.  It is non-zero.

\

(iii) $[1,n,1]$ is the unique element $P \in \textbf{P}^{BC}_n$ such that $A_iP=0,\forall i =1,2,\ldots,n$.

\

(iv)  The elements defined in (i) are all non-zero.

\

(v)   The elements described in $\textbf{P}^{BC}_n$ are distinct.

 \end {lemma}

 \begin {proof}  It is clear that the proposed expression for $\textbf{P}^{BC}_n$ contains the generators.  By $(16)$ it is stable under the BGG operators. Hence it must equal all of  $ \textbf{P}^{BC}_n$.

  The first part of (ii) obtains from (i) and our formula for length. The second part follows from Lemma \ref {6.5}.  (iii) follows from the relations in $(16)$. (iv) follows from (iii) and the second part of (ii).

  The relations in $(15)$ imply that for all $P,P^\prime \in \textbf{P}$ one has $P=P^\prime$ only if $\ell(P) = \ell (P^\prime)$ and Supp $P= \text{Supp} \ P^\prime$.  Fix $m$ a positive integer $\leq n$.  Thus it is enough to show that $[i,n,m-i]=[j,n,m-j]$ implies $i=j$. If not  we can assume $i<j$.  Apply $A_i$ to this equality.  If the right hand side is non-zero it must equal $[j,n,n-j-1]$, through length and the relations in $(16)$.  By this means we are reduced to the case $i=1$. Then comparison of supports forces $m-j=1$, so we obtain $[1,n,m-1]=[m-1,n,1]$.  If $m=2$, then $j=1$ which is a contradiction.  If $m>2$, apply $A_{m-2}$ to both sides.  Then $A_{m-2}[1,n,m-1]=0$, by $(16)$, which is again a contradiction.

 \end {proof}

\subsection{Types $B_n,C_n$}\label{6.6.1}

\

 Let us describe the map $\mathscr P$ in types $B_n$ and $C_n$.  These will be slightly different.  Use the Bourbaki notation \cite [Planches I-X]{B} for the roots.

 Consider first $B_n$.  It is convenient to use the positive roots of $C_n$ to describe the positive coroots of $B_n$.  These take the form $\varepsilon_i-\varepsilon_j:1\leq i <j \leq n$, which form the positive roots of a subsystem of type $A_{n-1}$ together with  $\varepsilon_i+\varepsilon_j:1\leq i\leq j\leq n$.

 Define $\mathscr P$ by
 $$\mathscr P(\varepsilon_j-\varepsilon_{i+1})=[i,j]:n-1\geq i \geq j \geq 1, \quad \mathscr P(\varepsilon_j+\varepsilon_i)=[i,n,j] :n\geq i \geq j \geq 1).$$

  By \cite [Planche II]{B} one has $\rho(\varepsilon_i)=n-i+1/2$ and so $\rho(\varepsilon_j-\varepsilon_{i+1})=i+1-j, \rho(\varepsilon_j+\varepsilon_i)=n-i+n-j+1$.  Consequently $\ell(\mathscr P(\gamma^\vee))=\rho(\gamma^\vee)$, for all $\gamma \in \Delta^+$.

 Consider $C_n$.  The positive coroots take the form $\varepsilon_i-\varepsilon_j:1\leq i <j \leq n$, which form the positive roots of a subsystem of type $A_{n-1}$ together with  $\varepsilon_i+\varepsilon_j:1\leq i < j\leq n, \varepsilon_i: i=1,2,\ldots,n$.

 Define $\mathscr P$ as before on the subsystem of type $A_{n-1}$.  In contrast we set
$$\mathscr P(\varepsilon_i+\varepsilon_j)=[i,n,j-1]:1\leq i < j \leq n), \quad \mathscr P(\varepsilon_i)=[i,n,n]:1\leq i \leq n).$$

By \cite [Planche III]{B} one has $\rho(\varepsilon_i)=n-i+1$ and so $\rho(\varepsilon_i+\varepsilon_j)=n-i+n-j+2$, for all $1\leq i < j \leq n$. As before we conclude that $\ell(\mathscr P(\gamma^\vee))=\rho(\gamma^\vee)$, for all $\gamma \in \Delta^+$.

The relations in $(16)$ translate to give

$$\begin {array}{rclrcl}
 A_n\mathscr P(\varepsilon_j-\varepsilon_n)&=&\mathscr P(\varepsilon_j+\varepsilon_n)&:&j<n, \quad &\text{type} \  B_n,\\
 A_n\mathscr P(\varepsilon_i)&=&\mathscr P(\varepsilon_i+\varepsilon_n)&:&i<n, \quad &\text {type} \  C_n, \\
 A_{i-1}\mathscr P(\varepsilon_j+\varepsilon_i)&=&\mathscr P(\varepsilon_j+\varepsilon_{i-1})&:&j<i, \quad &\text {type} \ B_n, \\
 A_{i-1}\mathscr P(\varepsilon_i+\varepsilon_{j+1})&=&\mathscr P(\varepsilon_{i-1}+\varepsilon_{j+1})&:&i \leq j, \quad &\text {type} \ C_n,\\
  A_j\mathscr P(\varepsilon_i+\varepsilon_{j+1})&=&\mathscr P(\varepsilon_i+\varepsilon_j)&:&i \leq j, \quad &\text {type} \ C_n,\\
  A_{j-1}\mathscr P(\varepsilon_j+\varepsilon_i)&=&\mathscr P(\varepsilon_{j-1}+\varepsilon_i)&:&j<i, \quad &\text {type} \ B_n. \\

 \end{array} \eqno{(17)}$$

 All other expressions vanish \textit{except} $A_{i-1}\mathscr P(2\varepsilon_i)=[i-1,n,i]:1<i\leq n$ in type $B_n$ and $A_n\mathscr P(\varepsilon_j-\varepsilon_n)=[n,n,j]:j<n$ in type $C_n$.  In these cases the left hand side does not lie in the image of $\mathscr P$.

We have the following analogue of the result described in \ref {6.4} for the simply-laced case.

\begin {lemma} (Types $B,C$).  Take $\alpha \in \pi, \gamma^\vee \in \Delta^{\vee+}$. Suppose $A_{\alpha^\vee} \mathscr P(\gamma^\vee) \in \im \mathscr P$ and $\alpha(\gamma^\vee)=-1$.  Then $A_{\alpha^\vee} \mathscr P(\gamma^\vee)=\mathscr P(s_\alpha\gamma^\vee)$.
\end {lemma}

\begin {proof}  This is a straightforward verification using $(17)$.
\end {proof}

\textbf{Remark 1}.   Notice that line (2) of $(17)$ and line $3$ of $(17)$ for $j=i-1$ do \textit{not} give a similar conclusion; but rather what is stated in Lemma \ref {6.9}(i).

\textbf{Remark 2}.  This result is illustrated for $B_4,C_4$ is Figures $1,2$.  One might compare this to the corresponding (much simpler result) in the simply-laced case illustrated in Figures $4,5$.

 \subsection{Type $F_4$}\label{6.7}

 \

In all cases studied so far the number of elements of the Zhelobenko monoid decreases with length.    Again for the unbranched Dynkin diagrams in all cases so far one may check from the above description of the Zhelobenko monoid, that the elements of the monoid may be obtained from a presentation of the unique largest element by successive cancelling of factors on both sides.  For example in type $B_3$, the unique element of length five is $1232(1)$, those of lengths four are $123(2),232(1)$, those of lengths three are $12(3),23(2),32(1)$, those of length two are $1(2)=2(1),2(3),3(2)$  and those of length one are $(1),(2),(3)$.

This procedure does not work in type $F_4$.   We can write the unique longest element (in type $F_4$) in the form $1234321323(4)$.  Whilst we can cancel off from the left, cancelling off from the right would give $123432132(3)$.  However this element is zero because the factor $32132(3)$ is an element of length $6$ in type $B_3$, whilst $m_\ell=5$.

Again Type $F_4$ is the unique case which does \textit{not} have the property that the number of elements in the Zhelobenko monoid decreases with length.

In view of the above we simply calculated the Zhelobenko monoid by brute force.  We found that the number of elements of length $i=1,2,\ldots,11$ is given by the following sequence $4,4,5,6,5,4,4,4,3,2,1$ which adds to $42$. In particular it admits a unique element of length $11$ which just happens to be the largest exponent in type $F_4$ ! (This element is non-zero by Lemma \ref {6.5}. The argument in Lemma \ref {6.6} was used to prove that the remaining elements are non-zero.) Moreover one checks that it is the unique element annihilated by the $A_i:i=1,2,3,4$.  One may remark that $42$ does not divide the order of the Weyl group and so even numerically the Zhelobenko does not identify with a Weyl group quotient (defined say by a stabilizer of a dominant weight) given its induced (weak left) Bruhat order though there are some superficial resemblances.

In addition to the above we found an injective map $\mathscr P$ of the set of positive coroots to the  Zhelobenko monoid and verified that it satisfied the conclusion of Lemma \ref {6.6.1}.  Unfortunately for technical reasons we could not include this data in Figure $3$ except that the image of $\mathscr P$ is given by the unencircled vertices. Thus we drew a second version of Figure $3$ designated as Figure $3^*$ in which this addition data is included and in which other data is omitted.  The reader needs to imagine these two sets of data superimposed.  Then the above results are easily verified using this presentation.

Risking repetition we summarize the above as the

\begin {lemma} ($\mathfrak g$ simple not of type $G_2$).  There exists an injective map $\mathscr P: \Delta^{\vee +} \rightarrow \textbf{P}$ satisfying
$A_{\alpha^\vee} \mathscr P(\gamma^\vee)=\mathscr P(s_\alpha\gamma^\vee)$, for all $\alpha \in \pi, \gamma^\vee \in \Delta^{\vee +}$, whenever the left hand side lies in the image of $\mathscr P$ \textbf{and} when $\alpha(\gamma^\vee)=-1$.

\end {lemma}

\subsection{}\label{6.8}

Suppose $\mathfrak g$ not of type $G_2$.

\begin {lemma}

\

(i) \ Suppose $\alpha \in \pi, \gamma \in \Delta^+$ are such that $\gamma^\vee \pm \alpha^\vee$ are both coroots.  Then $A_{\alpha^\vee}\mathscr P(\gamma^\vee)=\mathscr P(\alpha^\vee +\gamma^\vee)$.

\

(ii) \ Suppose $\alpha \in \pi, \gamma \in \Delta^+$ are such that neither $\gamma^\vee \pm \alpha^\vee$ are coroots.  Then  $A_{\alpha^\vee}\mathscr P(\gamma^\vee)=0$.
\end {lemma}

\begin {proof} The hypothesis of (i) implies that $\mathfrak g$ is not simply-laced and moreover that $\alpha^\vee,\gamma^\vee$ are short orthogonal coroots.  In particular $\gamma^\vee \pm \alpha^\vee$ are both long coroots.

In type $B_n$ only the $2\varepsilon_i:i=1,2,\ldots,n$ are long positive coroots, so there exists $i \in \{1,2,\ldots,n-1\}$ such that $\alpha^\vee= \varepsilon_i-\varepsilon_{i+1},\gamma^\vee =\varepsilon_i+\varepsilon_{i+1}$.  Then $A_{\alpha^\vee}\mathscr P(\gamma^\vee)=A_i[i+1,n,i]=[i,n,i] = \mathscr P(\alpha^\vee + \gamma^\vee)$, as required.

In type $C_n$ only the $\varepsilon_i:i=1,2,\ldots,n$ are short positive coroots, so $\alpha^\vee =\varepsilon_n$ and there exists $i \in \{1,2,\ldots,n-1\}$ such that $\gamma^\vee =\varepsilon_i$. Then $A_{\alpha^\vee}\mathscr P(\gamma^\vee)=A_n[i,n,n]=[i,n,n-1] = \mathscr P(\alpha^\vee + \gamma^\vee)$, as required.

In type $F_4$, using the \textit{reverse} labelling compared to \cite [Planche VIII]{B}, the short positive coroots take the form $\varepsilon_i:i=1,2,3,4, \frac {1}{2}(\varepsilon_1\pm \varepsilon_2 \pm \varepsilon_3 \pm \varepsilon_4)$.  This forces $\alpha^\vee = \varepsilon_4=\alpha_3^\vee$ and $\gamma^\vee = \varepsilon_i:i=1,2,4$.  Writing $\alpha_i^\vee:i=1,2,3,4$ as $a,b,c,d$, this means that $\alpha^\vee=c$ and $\gamma^\vee = b+c,a+b+c, a+2b+3c+d$.  Then the assertion can be read off from Figure $3$ and Figure $3^*$.

The hypothesis of (ii) implies that $\alpha(\gamma^\vee)=0$.  Thus the assertion holds by Lemma \ref {6.4}(ii) in the simply-laced case.  In types $B_n,C_n$ the assertion follows from the fact that $(17)$ together with the remarks following it give all non-vanishing relations.  In type $F_4$ the assertion can be read off from Figure $3$ and Figure $3^*$.
\end {proof}

\textbf{Remark}.  Suppose $\alpha \in \pi, \gamma^\vee \in \Delta^{\vee +}$ satisfy $\alpha(\gamma^\vee)=-2$.  Then $A_\alpha\mathscr P(\gamma^\vee)$ cannot equal $\mathscr P(s_\alpha \gamma^\vee)$ due to length considerations. In general the left hand side does not belong to the image of $\mathscr P$. (See the line following $(17)$ for example.)  This difficulty imposes additional mental gymnastics.  However apart from \ref {6.8}, \ref {7.4} and \ref {7.5} our analysis in quite intrinsic not relying on the details of any particular case.

\subsection{}\label{6.9}

We recall a simple fact which is true for all $\mathfrak g$ simple, though it usually stated for roots rather than coroots.

\begin {lemma}  Every positive coroot $\gamma^\vee$ can be written in the form
$s_{j_1}s_{j_2}\ldots s_{j_s}\alpha_{j_{s+1}}$ with $o(s_t \ldots s_{j_s}\alpha_{j_{s+1}})$ strictly decreasing with $t$.
\end {lemma}

\section{The $\mathscr P(\gamma^\vee):\gamma \in \Delta^+$ as Polynomials}

\subsection{}\label{7.1}

Assume $\mathfrak g$ not of type $G_2$.  In this section we introduce polynomials $P_{\gamma^\vee}:\gamma^\vee \in \Delta^{\vee +}$ which form a realization of the Zhelobenko elements $\mathscr P(\gamma^\vee)$. Eventually we show (Proposition \ref {8.2}) that they satisfy some rather simple recurrence relations $(24)$ when evaluated at $\rho$.  This is relatively straightforward in the simply-laced case and little tedious for types $B_n,C_n,F_4$.

 In Section $8$, the above results are used to establish the truth of the analogue Kostant problem.  Basically this involves some linear algebra combined with a knowledge of how exponents behave.  As already pointed out in \ref {C.2} the analogue Kostant problem is trivial in type $G_2$ so it is unnecessary to adapt the present computations to include this case.  Nevertheless the key result (Proposition \ref {8.2}) easily adapts to this case for the trivial reason that $\{\gamma^\vee \in \Delta^{\vee+}|o(\gamma^\vee)=r\}$ has cardinality $\leq 1$ for $r\geq 2$ if $\mathfrak g$ has rank $\leq 2$.

Thus in the remainder of this section and all but the final part (namely \ref {8.5}) of Section $8$ we shall assume that $\mathfrak g$ is simple but not of type $G_2$.

\subsection{}\label{7.2}

Let $x_{\alpha^\vee}:\alpha \in \Delta$ be the root vectors in a Chevalley basis for $\mathfrak g^\vee$.

Suppose that $\alpha^\vee, \beta^\vee, \alpha^\vee+\beta^\vee$ are non-zero coroots. Then in standard notation we may write
$$[x_{\alpha^\vee},x_{\beta^\vee}]=
N_{\alpha^\vee,\beta^\vee}x_{\alpha^\vee+\beta^\vee},$$
where the coefficients $N_{\alpha^\vee,\beta^\vee}$ are integers (just $\pm 1$ in the simply-laced case). We shall use the convention that $N_{\alpha^\vee,\beta^\vee}=0$ if any one of the three elements $\alpha^\vee, \beta^\vee, \alpha^\vee+\beta^\vee$ is not a coroot.

Take $\alpha, \beta, \delta \in \pi, \gamma \in \Delta^+$, in (i)-(v) below.

Given $\alpha(\gamma^\vee)<0$, then $\gamma^\vee-\alpha^\vee$ is not a coroot, since we have excluded $\mathfrak g$ of type $G_2$.

Via the Jacobi identity we obtain

\

(i) If $\alpha(\gamma^\vee)<0$, then $N_{\alpha^\vee,\gamma^\vee}N_{-\alpha^\vee, \alpha^\vee+\gamma^\vee}=-\alpha(\gamma^\vee)$.

\

Similarly

 \

(ii) If $\alpha(\gamma^\vee)>0$, then $N_{-\alpha^\vee,\gamma^\vee}N_{\alpha^\vee, -\alpha^\vee+\gamma^\vee}=\alpha(\gamma^\vee)$.

\

(iii) If $\alpha(\gamma^\vee)=0$, then both $N_{\alpha^\vee,\gamma^\vee}$ and $N_{-\alpha^\vee,\gamma^\vee}$ are zero unless both $\alpha^\vee$ and $\gamma^\vee$ are short coroots.

\

Again

\

(iv) Suppose $\alpha,\beta$ are distinct. Then $N_{\alpha^\vee,\gamma^\vee}N_{-\beta^\vee, \alpha^\vee+\gamma^\vee}=N_{-\beta^\vee,\gamma^\vee}N_{\alpha^\vee, \gamma^\vee-\beta^\vee}$.

\

(v)  Suppose $\alpha,\beta$ are distinct and that $\alpha^\vee+\gamma^\vee$ is a coroot, but $\alpha^\vee+\beta^\vee+\gamma^\vee$ is not a coroot.  Then $N_{-\delta^\vee,\gamma^\vee}N_{\alpha^\vee, \gamma^\vee-\delta^\vee}N_{\beta^\vee,\alpha^\vee +\gamma^\vee-\delta^\vee}=0$, unless $\beta=\delta$ in which case it equals $N_{\alpha^\vee,\gamma^\vee}\beta(\alpha^\vee+\gamma^\vee).$

\subsection{}\label{7.3}

Given $\gamma \in \Delta$, then $\gamma$ is a long (resp. short) root exactly when $\gamma^\vee$ is a short (resp. long) coroot.

In $\mathfrak g$ is simply-laced we declare all roots to be short and all coroots to be long.

Take $\alpha,\gamma \in \Delta$ such that $\alpha(\gamma^\vee) \neq 0$.  Then $|(\alpha(\gamma^\vee)|=1$ fails exactly when $\alpha,\gamma^\vee$ are both long.

More specifically we take $\alpha \in \pi,\gamma^\vee \in \Delta^{\vee+}$ in the above and we call this pair ``good" if $\alpha(\gamma^\vee)=-1$ and ``bad" if $\alpha(\gamma^\vee)=-2$ \textit{and} if some additional condition is satisfied (\ref {7.4}.

Our analysis is straightforward in the good case, with some minor complications in the bad case.

\subsection{}\label{7.4}

If $\alpha:=\alpha_i \in \pi$ we set $A_{\alpha^\vee}=A_i,P_{\alpha^\vee}=P_i$.  If $\alpha \in \pi,\gamma^\vee \in \Delta^{\vee+}$ is a good pair and $\alpha(\gamma^\vee)=-1$, we define
$$A_{\alpha^\vee}P_{\gamma^\vee}:=N_{\alpha^\vee,\gamma^\vee}
P_{\alpha^\vee+\gamma^\vee}. \eqno {(18)}$$

 Suppose that $\alpha \in \pi,\gamma^\vee \in \Delta^{\vee+}$ is a bad pair.  Then $2\alpha^\vee+\gamma^\vee$ is a coroot and we must define $P_{2\alpha^\vee+\gamma^\vee}$.  It can happen that it is defined by the previous rule (by taking a different path in the Zhelobenko graph - as per Figures).  The new rule which we propose below is only applied in two circumstances and in particular when no alternative paths exist.

 \

 \textbf{The starting bad pair}.

 \

 This is when both $\alpha$ and $\gamma$ are simple roots. In this case we set

$$A_{\gamma^\vee}A_{\alpha^\vee} P_{\gamma^\vee}=\frac{\alpha(\gamma^\vee)N_{\gamma^\vee,
\alpha^\vee}}{N_{-\alpha^\vee,\gamma^\vee+2\alpha^\vee}} P_{2\alpha^\vee+ \gamma^\vee}. \eqno {(19)}$$

Notice that here $\alpha^\vee$ is a unique simple coroot $\delta^\vee$ such that $2\alpha^\vee+ \gamma^\vee-\delta^\vee$ is a coroot.

We remark that there is a starting bad pair for every non-simply-laced $\mathfrak g$ (outside $G_2$).

\

\textbf{The intermediate bad pairs}.

\

This is when $\alpha^\vee$ is a unique simple coroot $\delta^\vee$ such that $2\alpha^\vee+ \gamma^\vee-\delta^\vee$ is a coroot, but $\gamma$ is not simple.
In this case there is a neighbour $\beta\neq \gamma$ of $\alpha$ in the Dynkin diagram. Then $\beta(\alpha^\vee)=-1$, since $\alpha$ is a long root.  Then $\beta(2\alpha^\vee+\gamma^\vee)\geq 0$, since $\beta^\vee\neq \alpha^\vee$.  Since we have excluded type $G_2$, this forces $\beta(\gamma^\vee)=2$.  Yet $\beta\neq \gamma$, so $\beta$ is also a long root (which is important and rather curious).

We conclude that there are no intermediate bad pairs in type $C$.  The intermediate bad pairs in type $B_n$ are given by the pairs $(\varepsilon_{i-1}-\varepsilon_i,2\varepsilon_i:i=2,3,\ldots,n-1$ with $\beta=\varepsilon_i-\varepsilon_{i+1}$.

Suppose that $\alpha \in \pi,\gamma^\vee \in \Delta^{\vee+}$ is an intermediate bad pair in type $F_4$.  Then by the above $\alpha$ is the unique long simple root with just one root (necessarily long) as a neighbour and this neighbour is $\beta$.  So as to use \cite [Planche VIII]{B} for the \textit{coroots} we shall label the simple roots in the reverse sense compared to Bourbaki.  Thus $\alpha=\alpha_4,\beta =\alpha_3$.  One checks that there are just two possible choices of $\gamma^\vee$, namely $\alpha^\vee_2+2\alpha^\vee_3,\alpha^\vee_1+2\alpha^\vee_2+
2\alpha^\vee_3+2\alpha^\vee_4$.  (We do not need to know the intermediate bad pairs explicitly or even their number, though all this would be a bit futile if there were not any!)

Finally if $\alpha \in \pi,\gamma^\vee \in \Delta^{\vee+}$ is an intermediate bad pair we set
$$A_{\beta^\vee}A_{\alpha^\vee} P_{\gamma^\vee}=\frac{\beta
(\alpha^\vee+\gamma^\vee)N_{\alpha^\vee,\gamma^\vee}}{N_{-\alpha^\vee,\alpha^\vee+\gamma^\vee}} P_{2\alpha^\vee+ \gamma^\vee}. \eqno {(20)}$$

\subsection{}\label{7.5}

We need the preliminary

\begin {lemma} Take $\alpha \in \pi, \gamma \in \Delta^+$.

\

(i)  Suppose $\alpha,\gamma^\vee$ is a good pair.  Then $A_{\alpha^\vee}\mathscr P(\gamma^\vee) = \mathscr P(\alpha^\vee+\gamma^\vee)$.

\

(ii)  Suppose $\alpha,\gamma^\vee$ is a bad pair and take either $\beta=\gamma$ (for the starting bad pair) or $\beta$ as in $(19)$ for an intermediate bad pair). Then $A_{\beta^\vee}A_{\alpha^\vee}\mathscr P(\gamma^\vee) = \mathscr P(2\alpha^\vee+\gamma^\vee)$.

\end {lemma}

\begin {proof}   If $\mathfrak g$ is simply-laced, the assertions follow from Lemma \ref {6.4}(i).

Consider (ii) in the non-simply-laced cases.

In type $B_n$, the only elements of the image of $\mathscr P$ pushed out of the image of $\mathscr P$ by applying some BGG operator are the $[i,n,i]$ which are the images of the long coroots $2\varepsilon_i: i=1,2,\ldots,n$ and indeed exactly by $A_{i-1}$, where in this : $i=2,3,\ldots,n$.  This corresponds exactly to the bad pair $\alpha=\alpha_{i-1}, \gamma^\vee=2\epsilon_i$.  Moreover in this $A_{\beta^\vee}=A_i$ and $A_iA_{i-1}[i,n,i]=[i-1,n,i-1]=\mathscr P(2\epsilon_{i-1})=\mathscr P(2\alpha^\vee+\gamma^\vee)$, as required.

In type $C_n$, there is just one element of the image of $\mathscr P$ pushed out of the image of $\mathscr P$ by applying some BGG operator.  It is $(n-1)$ with the BGG operator being $A_n$. This corresponds exactly to the starting bad pair $\alpha=\alpha_n, \gamma^\vee=\epsilon_{n-1}-\epsilon_n$.  Yet $A_{n-1}A_n((n-1))=[n-1, n,n-1]=\mathscr P(\epsilon_{n-1}+\epsilon_n)=\mathscr P(2\alpha^\vee+\gamma^\vee)$, as required.

In type $F_4$ one checks from Figures $3$ and $3^*$, that $A_{\alpha^\vee}\mathscr P(\gamma^\vee) \notin \im \mathscr P$ only when $\alpha$ and $\gamma^\vee$ are both long.  However unlike the classical case, the pair $\alpha,\gamma^\vee$ is not always a bad pair in our sense.  In the three cases that is a bad pair one checks that (ii) and (iii) hold again by inspection of Figures $3$ and $3^*$.

This proves (ii).  Moreover we have also shown that if $\alpha,\gamma^\vee$ is a good pair, then $A_{\alpha^\vee}\mathscr P(\gamma^\vee) \in \im \mathscr P$.  Consequently (i) follows from Lemma \ref {6.8}.

\end {proof}

\subsection{}\label{7.6}

\begin {lemma} Take $\alpha,\beta \in \pi$ and suppose that $\gamma^\vee-2\alpha^\vee,\gamma^\vee-2\beta^\vee,\gamma^\vee$ are all positive coroots.  The $\alpha=\beta$.

\end {lemma}

\begin {proof} Under the hypothesis these must all be long coroots, with $\gamma^\vee-\alpha^\vee$ and $\gamma^\vee-\beta^\vee$ short coroots. Again $\alpha(\gamma^\vee)=2$ and
$\beta(\gamma^\vee)=2$.  If $\beta\neq \alpha$, then $\beta(\gamma^\vee-2\alpha^\vee)=2-2\beta(\alpha^\vee)\geq 2$, which forces $\beta(\alpha^\vee)=0$. Then $\beta(\gamma^\vee-\alpha^\vee)=2$, implying that $\gamma^\vee-\alpha^\vee$ is a long coroot.  This contradiction proves the lemma.
\end {proof}

\subsection{}\label{7.7}

Combining Lemmas \ref {6.9} and \ref {7.6} we deduce that every positive coroot $\gamma^\vee$ can be written in the form $s_{j_1}s_{j_2}\ldots s_{j_s}\alpha^\vee_{j_{s+1}}$, with $\alpha_{j_t}, s_{j_{t+1}}\ldots s_{j_s}\alpha^\vee_{j_{s+1}}$ a good or a bad pair for $t=1,2,\ldots,s$.

This and Lemma \ref {7.5} have the following consequence

\begin {lemma}

\

\smallskip

(i) \ In the $P_{\gamma^\vee}$ constructed through $(18),(19),(20)$ all the positive coroots appear.

\

(ii)\  For all $\gamma^\vee \in \Delta^{\vee +}$, $P_{\gamma^\vee}$ is a multiple of $\mathscr P(\gamma^\vee)$ as elements of $k\textbf{P}$.

\end {lemma}

\textbf{Remark}.  There can be several different paths to obtain $P_\gamma$. It is not immediate that different paths give the same scalar.  However this will be shown in \ref {7.8} below.

\subsection{}\label{7.8}

\begin {prop}

\

\smallskip

(i) For all $\gamma \in \Delta^+ \setminus \pi$, one has
$$P_{\gamma^\vee} =\frac {\sum_{\alpha \in \pi}N_{-\alpha^\vee,\gamma^\vee} P_{\gamma^\vee-\alpha^\vee}}{1+\gamma^\vee}.$$

\

(ii) For all $\alpha \in \pi, \gamma \in \Delta^+$, with $\alpha(\gamma^\vee) > 0$, one has $A_{\alpha^\vee} P_{\gamma^\vee}=0$.

\

(iii)  For all $\alpha \in \pi, \gamma \in \Delta^+$, with $\alpha(\gamma^\vee) = 0$, one has $A_{\alpha^\vee} P_{\gamma^\vee}=0$, unless $\gamma^\vee +\alpha^\vee$ is a coroot.

\end {prop}

\begin {proof}  The assertions are proved successively via induction on $\rho(\gamma^\vee)=o(\gamma^\vee)$.

Consider (ii),(iii) and suppose $o(\gamma^\vee)=1$.  Then either $\alpha = \gamma$ or $\alpha(\gamma^\vee)=0$.  In either case our assertion follows from $(15)$.

Now assume that $r:=o(\gamma^\vee) >1$ and that (ii),(iii) has been proved for $o(\gamma^\vee)=r-1$ and (i) when $o(\gamma^\vee)=r$.

Suppose $o(\gamma^\vee)=r$ and consider (ii),(iii).  By (i) and the induction hypothesis
$$(1+\gamma^\vee) P_{\gamma^\vee}=\sum_{\beta \in \pi}N_{-\beta^\vee,\gamma^\vee} P_{\gamma^\vee-\beta^\vee}.\eqno {(21)}$$

Under the hypothesis of (ii), the inequality $\alpha(\gamma^\vee-\beta^\vee) > 0$ results, unless $\alpha=\beta$. Hence $A_{\alpha^\vee} P_{\gamma^\vee-\beta^\vee}=0$, by the induction hypothesis if $\alpha \neq \beta$. The applying $A_{\alpha^\vee}$ to both sides of $(21)$ we obtain

$$A_{\alpha^\vee}(1+\gamma^\vee)P_{\gamma^\vee}=N_{-\alpha^\vee,\gamma^\vee}
A_{\alpha^\vee} P_{\gamma^\vee-\alpha^\vee}=N_{-\alpha^\vee,\gamma^\vee}N_{\alpha^\vee,\gamma^\vee
-\alpha^\vee}P_{\gamma^\vee}=\alpha
(\gamma^\vee)P_{\gamma^\vee},\eqno {(22)}$$
via (i) of \ref {7.2}.  On the other hand by $(13)$ the left hand side of $(22)$ equals $(1+s_\alpha\gamma^\vee)A_{\alpha^\vee} P_{\gamma^\vee} +\alpha(\gamma^\vee)P_{\gamma^\vee}$.  Hence $A_{\alpha^\vee} P_{\gamma^\vee}=0$, as required.

Under the hypothesis of (iii) if $\gamma^\vee\pm \alpha^\vee$ are not coroots, then $A_{\alpha^\vee}P_{\gamma^\vee}=0$, by Lemma \ref {7.7}(ii) and Lemma \ref {6.8}(ii).

Thus (ii) and (iii) have been extended to the case $o(\gamma^\vee)=r$.

 Suppose that $o(\gamma^\vee)=r$ and show that (i) holds for $o(\gamma^\vee)=r+1$. For this we must examine three cases.
 
 \textbf{Case 1}.

 Let $\alpha , \gamma^\vee$ be a good pair (so then $\alpha(\gamma^\vee)=-1$).  In particular $\gamma^\vee-\alpha^\vee$ is not a coroot.

 Apply $A_{\alpha^\vee}$ to both sides of $(21)$. Using $(13)$ this gives

$$(1+s_\alpha(\gamma^\vee))A_{\alpha^\vee} P_{\gamma^\vee}+\alpha(\gamma^\vee)P_{\gamma^\vee}=\sum_{\beta \in \pi}N_{-\beta^\vee,\gamma^\vee}A_{\alpha^\vee} P_{\gamma^\vee-\beta^\vee}.\eqno {(23)}$$

 Assume that $N_{-\beta^\vee,\gamma^\vee}\neq 0$ in the above sum. By the observation made above we cannot have $\beta = \alpha$ in the sum (so then $\alpha(\gamma^\vee-\beta^\vee)\geq -1$) and we may assume $\alpha(\gamma^\vee-\beta^\vee) \leq 0$ by (ii) and the induction hypothesis.

Suppose $\alpha(\gamma^\vee-\beta^\vee)=0$ and $A_{\alpha^\vee}P_{\gamma^\vee - \beta^\vee} \neq 0$.  Then by Lemmas \ref {6.8}(ii), \ref {7.7}(ii), and (iii) combined with the induction hypothesis, it follows that $\gamma^\vee-\beta^\vee\pm \alpha^\vee$ are both long coroots. In particular $\mathfrak g$ cannot be simply-laced.

 In addition $\alpha(\gamma^\vee-\beta^\vee-\alpha^\vee)=-2$, so $\alpha$ is a long root. Yet $\alpha(\gamma^\vee)=-1$, so $\gamma$ is also a long root and $\gamma(\alpha^\vee)=-1$.  Again $\alpha(\beta^\vee)=-1$, so $\beta$ is a long root and $\beta(\alpha^\vee)=-1$.

 By the above $\gamma^\vee-\beta^\vee$ is a short coroot. Yet $\gamma^\vee-\beta^\vee-\gamma^\vee$ is a coroot, so $\gamma(\gamma^\vee-\beta^\vee)=1$ and hence $\gamma(\beta^\vee)=1$, since $\beta$ is a long root. Thus $\beta(\gamma^\vee)=1$.  Consequently $\beta(\gamma^\vee-\beta^\vee-\alpha^\vee)=0$.  Since $\gamma^\vee-\alpha^\vee$ is a coroot, it follows that $\gamma^\vee-\beta^\vee-\alpha^\vee$ is a short coroot, in contradiction to what we had previously shown.

 We conclude from the above that $A_{\alpha^\vee}P_{\gamma^\vee-\beta^\vee}$ in the right hand side of $(23)$ can only be non-zero if $\alpha(\gamma^\vee-\beta^\vee)=-1$, that is to say  $\alpha,\gamma^\vee-\beta^\vee$ is a good pair. Then by $(18)$ and \ref {7.2}(iv), it follows that the right hand side of $(23)$ can be written as
$$N_{\alpha^\vee,\gamma^\vee}(\sum_{\beta \in \pi|\beta \neq \alpha}N_{-\beta^\vee,\alpha^\vee+\gamma^\vee}P_{\gamma^\vee-\beta^\vee +\alpha^\vee}).$$

On the other hand by \ref {7.2}(i), the second term in the left hand side of $(23)$ taken to the right hand side would be the term in the above sum had we permitted $\beta=\alpha$.  Thus from $(23)$ we obtain

$$
{A_\alpha P_\gamma}  =
\frac{N_{\alpha^\vee,\gamma^\vee}}{1+s_\alpha\gamma^\vee}[\sum_{\beta \in \pi}N_{-\beta^\vee,\alpha^\vee+\gamma^\vee}P_{\gamma^\vee+\alpha^\vee-\beta^\vee}].
$$
 Through $(18)$ and cancelling out the non-zero scalar $N_{\alpha^\vee,\gamma^\vee}$, this gives (i) with $o(\gamma)=r+1$.  This concludes Case 1.

  Let $\alpha , \gamma^\vee$ be a bad pair (so then $\alpha(\gamma^\vee)=-2$).

  \textbf{Case 2}. 
   
   Let $\alpha , \gamma^\vee$ is a starting bad pair, that is $\gamma \in \pi$.
   
    Since $\gamma(\alpha^\vee)=-1$ and $\gamma \in \pi$, it follows that $\gamma,\alpha^\vee$ is a good pair, whilst $\gamma(2\alpha^\vee+\gamma^\vee)=0$.  In view of $(9)$ and $(18)$ we obtain
  $$A_{\gamma^\vee}A_{\alpha^\vee}P_{\gamma^\vee}=\frac{
  \alpha(\gamma^\vee)N_{\gamma^\vee,\alpha^\vee}P_{\gamma^\vee+
  \alpha^\vee}}{1+\gamma^\vee+2\alpha^\vee}.$$

  Substituting from $(19)$ and cancelling out the non-zero scalar $\alpha(\gamma^\vee)N_{\gamma^\vee,\alpha^\vee}$ gives (i). This concludes Case 2.
  
  \textbf{Case 3}.

  Let $\alpha, \gamma^\vee$ be an intermediate bad pair.

   In this case we let $\delta \in \pi$ be the summation variable in the right hand side of $(21)$ and define $\beta$ as in \ref {7.4}. (Hopefully the reader can adjust to this change!)  Then $\alpha$ and $\beta$ which of course are now both fixed are both long roots and moreover $\beta(\alpha^\vee)=-1$, by \ref {7.4}.

  We need to show that $P_{2\alpha^\vee+\gamma^\vee}$ satisfies (i) when $o(2\alpha^\vee+\gamma^\vee)=r+1$, so we can assume $o(\gamma^\vee)=r-1$ in $(21)$.

  Apply $A_{\beta^\vee}A_{\alpha^\vee}$ to both sides of $(21)$. Assume that $A_{\beta^\vee}A_{\alpha^\vee} P_{\gamma^\vee-\delta^\vee} \neq 0$, for some term in the sum. Then by (ii) and the induction hypothesis  $\alpha(\gamma^\vee-\delta^\vee)=:a\leq0$.

  If $a=-2$, then $2\alpha^\vee+\gamma^\vee-\delta^\vee$ is a positive coroot, forcing $\delta=\alpha$ through the definition of a bad pair.  However in this case $\alpha(\gamma^\vee-\delta^\vee)=-4$, which is impossible.

  If $a=0$, then (just as we have saw in Case 1) it follows by the induction hypothesis that both $\gamma^\vee-\delta^\vee\pm \alpha^\vee$ are coroots.  Then $\gamma(\gamma^\vee-\delta^\vee - \alpha^\vee)=4-\gamma(\delta^\vee)$.  This forces to be $\gamma$ long, whilst by definition of a bad pair it is short.

  Hence $a=-1$. Consequently $\alpha,\gamma^\vee-\delta^\vee$ is a good pair.

  The equality $a=-1$ implies since $\alpha$ is long, that $\gamma^\vee-\delta^\vee$ is a short coroot.
 
  Since $\alpha^\vee$ is also a short coroot and $a=-1$, we conclude that $\gamma^\vee-\delta^\vee + \alpha^\vee$ is a short coroot. 
  
  The hypothesis $A_{\beta^\vee}A_{\alpha^\vee} P_{\gamma^\vee-\delta^\vee} \neq 0$, (ii),(iii) and the induction hypothesis implies that $b:=\beta(\gamma^\vee-\delta^\vee + \alpha^\vee)\leq 0$.  Thus $b \in \{0, -1\}$, by the previous paragraph.

  Suppose $b=0$. Since $a=-1$, we obtain $\alpha(\delta^\vee)=-1$. Yet $\beta(\alpha^\vee)=-1$ also and since these are all simple roots, we conclude that $\beta(\delta^\vee)=0$ or $\beta=\delta$. Thus $\beta(\gamma^\vee)=1$ or $\beta(\gamma^\vee)=3$.  Yet in \ref {7.4}, we showed that  $\beta(\gamma^\vee)=2$.
   
  Hence $b=-1$. Consequently $\beta,\gamma^\vee-\delta^\vee+\alpha^\vee$ is a good pair.

  Recall that again that $\beta(\gamma^\vee)=2$.  Thus $\alpha^\vee+\beta^\vee+\gamma^\vee$  cannot be a coroot.

  We may now compute the result of applying $A_{\beta^\vee}A_{\alpha^\vee}$ to the right hand side of $(21)$ using $(18)$ and \ref {7.2}(v).  The resulting expression is just $$N_{\alpha^\vee,\gamma^\vee}\beta(\alpha^\vee+\gamma^\vee)
  P_{\alpha^\vee+\gamma^\vee}.$$

  On the other hand $A_{\beta^\vee}A_{\alpha^\vee}$ applied to the left hand side of $(21)$ gives
  $$A_{\beta^\vee}((2\alpha^\vee+\gamma^\vee)
  A_{\alpha^\vee}P_{\gamma^\vee}-2P_{\gamma^\vee})=(2\alpha^\vee+\gamma^\vee)
 A_{\beta^\vee} A_{\alpha^\vee}P_{\gamma^\vee},$$
 since $\beta(2\alpha^\vee+\gamma^\vee)=0$, and using (ii) with the induction hypothesis to conclude that $A_{\beta^\vee}P_{\gamma^\vee}=0$.

 Comparing these last two displayed formulae and using $(20)$ gives (i) in Case (iii).

\end {proof}

\textbf{Remark}. \textit{Notice that (i) justifies the remark in \ref {7.7} concerning independence of scalars.}

\section{Proof of the The Analogue Kostant Conjecture}

\subsection{}\label{8.1}

One may remark in the above that the value of $A_\alpha (c+\gamma^\vee)$ is a scalar independent of $c \in k$.  This has the consequence if we start from $(12)$ instead of $(9)$ and similarly define $P^0_{\gamma^\vee}:\gamma \in \Delta^+$, then we obtain exactly as in Proposition \ref {7.8} the following

\begin {prop}

 For all $\gamma \in \Delta^+ \setminus \pi$, one has
$$P^0_{\gamma^\vee} =\frac {\sum_{\alpha \in \pi}N_{-\alpha^\vee,\gamma^\vee} P^0_{\gamma^\vee-\alpha^\vee}}{\gamma^\vee}.$$

(ii) For all $\alpha \in \pi$, $\gamma \in \Delta^+$, such that $\alpha^\vee+\gamma^\vee$ is not a coroot, one has $A_{\alpha^\vee} P^0_{\gamma^\vee}=0$.

\end {prop}

\textbf{Remark 1}.  Of course (ii) above is just the analogue of (ii) and (iii) of Proposition \ref {7.8} combined.  Again $P_{\gamma^\vee}^0$ is just the leading term of $P_{\gamma^\vee}$ and so this result can be deduced from Proposition \ref {7.8}.

\

\textbf{Remark 2}.   It will be useful in what follows to note that $2p_i^0=A_iq_i^0$.  This complements the above results and shows that the $P^0_{\gamma^\vee}$ are obtained from the partial derivatives of the invariants of applying products of the $A_i:i\in I$.  Moreover if we start from a generator of $S(\mathfrak h)^W$ of degree $m+1$, then these partial                                                                                                                                                                                                                                                                                                                                                                                                                                                                                                                                                derivatives form a basis for a copy of the adjoint module occurring in the degree $m$ component of $Q$.

\subsection{}\label{8.2}

Retain the hypotheses of \ref {C.2}.  It is clear that the degree of $P_{\gamma^\vee}$ is just $m-\rho(\gamma^\vee)$.  By this expression taking a strictly negative value we just mean that $P_{\gamma^\vee}=0$.  If $m<m_\ell$, then in particular $P_{\gamma^\vee}=0$ when $\rho(\gamma^\vee)=m+1$ and then (i) of Proposition \ref {7.8} gives relations on the scalars $P_{\gamma^\vee}:\rho(\gamma^\vee)=m$.  For example if $m=m_\ell$, then $o(\gamma^\vee)=m$ implies that $\gamma^\vee$ is the unique highest coroot $\beta^\vee_0$ and we can assume without loss of generality that $P_{\beta^\vee_0}=1$.

Before going further let us describe roughly how our proof of the analogue Kostant conjecture will proceed.

 We may use the conclusion of Proposition \ref {7.8}(i) to compute (partially) the values $P_{\gamma^\vee}(s\rho)$ by decreasing induction on $\rho(\gamma^\vee)$.  Similarly we may use the conclusion of Proposition \ref {8.1}(i) to compute (partially) the values $P^0_{\gamma^\vee}(s\rho)$ by decreasing induction on $r:=\rho(\gamma^\vee)$.  Since the only difference in the corresponding expressions comes from the denominator (being  $1+sr$ in the first case, and $sr$ in the second case) the new expressions should be related to the previous ones by the ratios of these common factors.  Written our explicitly this means that we should obtain the following induction relation on the ratios of $n$-tuples of elements of $k$, namely

$$\begin {array}{rcl}
 (P_{\gamma^\vee}(s\rho))_{\rho(\gamma^\vee)=r}
=c_r(P^0_{\gamma^\vee}
(s\rho))_{\rho(\gamma^\vee)=r}&\Rightarrow& \\\\
(P_{\gamma^\vee}(s\rho))_{\rho(\gamma^\vee)r-1}&=&
c_r\frac{1+sr}{sr}(P^0_{\gamma^\vee}(s\rho))_
{\rho(\gamma^\vee)=r-1}\\

 \end{array} \eqno{(24)}.$$

In this we shall consider as part of our induction hypothesis that $c_r$ on the left hand side is defined and then set $c_{r-1}=c_r \frac{1+sr}{sr}$, when $(24)$ is established.  Since $P_{\gamma^\vee} =P^0_{\gamma^\vee}$ when $o(\gamma^\vee)=m$ (because then  $P^0_{\gamma^\vee}$ is the leading term of $P_{\gamma^\vee}$ which is a scalar), our reverse induction starts.

Had it been the case that all these partial computations were complete then we could conclude that $(24)$ holds for all $r$.  This implies that the condition described in \ref {3.6} is satisfied and with it the truth of the analogue Kostant conjecture (for the simply-laced case).

Of course this partial computation is \textit{not} complete at any given step because the matrix occurring in the above, that is to say the matrix $M_r$ with entries $N_{-\alpha^\vee,\gamma^\vee} : \alpha \in \pi, \gamma^\vee \in \Delta^{\vee +}_r:=\{\delta^\vee \in \Delta^{\vee+} | o(\delta^\vee)=r\}$, has in general too small a rank.  More precisely its rank will in general be strictly less than $n_{r-1}$, where $n_r:=|\gamma \in \Delta^{\vee +}_r|$.  Nevertheless we shall prove the 

\begin {prop} Equation $(24)$ can be made to hold by adding to the $P_{\gamma^\vee}:\gamma^\vee \in \Delta^{\vee +}$ strictly lower order terms of the same form (that is to say coming from Zhelobenko invariants of strictly lower degree).
\end {prop}

\subsection{}\label{8.2.5}

\begin {lemma}  $M_r$ has rank equal to $n_r$.
\end {lemma}

\begin {proof}  Consider
$$\sum_{\alpha^\vee \in \pi}x_{-\alpha}.$$

It is a principal nilpotent element which we may identify with $f^\vee$.  Now
$$[f^\vee,x_\gamma^\vee]=\sum_{\alpha \in \pi}N_{-\alpha^\vee,\gamma^\vee}x_{\gamma^\vee-\alpha^\vee}.$$

Then the required assertion follows from the well-known consequence of $\mathfrak {sl}(2)$ theory which implies that $\ad f^\vee$ is an injection of
$\oplus_{\gamma \in \Delta^{\vee +}_r}x_{\gamma^\vee}$
into $\oplus_{\gamma \in \Delta^{\vee +}_{r-1}}x_{\gamma^\vee}$, for all integer $r\geq 1$.

\end {proof}

\textbf{Remark}.  In particular the $n_r: r \in \mathbb N^+$ are decreasing.

\subsection{}\label{8.3}

Because the $n_r:r \in \mathbb N^+$ are decreasing, they form an ordered partition $n_1,n_2,\ldots,$ of $|\Delta^+|$.  Let $n^*_1,n^*_2,\ldots,$ be the dual partition. As is well-known, Kostant \cite {K} proved (the conjecture of Shapiro) that the exponents are given by $m_i=n^*_i:i \in I$, that is to say the exponents form the dual partition.  In other words the number of exponents taking the value $r-1:r>1$ is just $n_{r-1}-n_r$, which is exactly the ambiguity in our proposed determination of the $P_i(s\rho)$ at the $r^{th}$ level. Here we recall that the exponents for $\mathfrak g$ and its Langlands dual are the same since their Weyl groups coincide.

From the above it follows that we have exactly $n_{r-1}-n_r$ generators $S(\mathfrak h)^W$ in degree $r$.  By the KNV theorem (see \ref {2.3}) they give rise to $n_{r-1}-n_r$ generators of the Zhelobenko invariants and hence to $n_{r-1}-n_r$ linearly independent $|I|$-tuples $(P_{(j),i})_{i \in \pi}:j=1,2,\ldots,n_{r-1}-n_r$ which are solutions of $(9)$ having common degree $r-2$ and whose leading terms $(P^0_{(j),i})_{i \in \pi}:j=1,2,\ldots,n_{r-1}-n_r$ are obtained from the differentials of generators
of the polynomial algebra $S(\mathfrak h)^W$, the latter being homogeneous of degree $r$.  They lead via Proposition \ref {8.1} to $|\Delta^{\vee +}|$-tuples $(P_{(j),\gamma^\vee})_{\gamma \in \Delta^{\vee +}}:j=1,2,\ldots,n_{r-1}-n_r$ and in particular to $n_{r-1}-n_r$ tuples of scalars $(P_{(j),\gamma^\vee})_{\gamma^\vee \in \Delta^{\vee+}_{r-1}}:j=1,2,\ldots,n_{r-1}-n_r$.

\begin {lemma} The scalars $(P_{(j),\gamma^\vee})_{\gamma^\vee \in \Delta^{\vee +}_{r-1}}:j=1,2,\ldots,n_{r-1}-n_r$ are exactly determined by the vanishing of the $P_{(j),\gamma^\vee}:j=1,2,\ldots,n_{r-1}-n_r, \gamma^\vee \in \Delta^{\vee +}_r$,  in other words as the kernel of $M_r$.
\end {lemma}

\begin {proof} That these $|\Delta^{\vee +}_{r-1}|$-tuples of scalars lie in $\ker M_r$ is already immediate from Proposition \ref {7.8} and the definition of $M_r$. Since $\ker M_r$ has dimension $n_{r-1}-n_r$ our assertion amounts to showing that these $(n_{r-1}-n_r)$-tuples are linearly independent.

 Consider a non-trivial linear combination of the $|\Delta^{\vee +}|$-tuples $$P=\sum_{j=1}^{n_{r-1}-n_r}d_j(P_{(j),\gamma^\vee})_{\gamma^\vee \in \Delta^{\vee +}}:d_j \in k,$$ such that  its entries $P_{\gamma^\vee}:=\sum_{j=1}^{n_{r-1}-n_r}d_jP_{(j),\gamma^\vee}$, are zero for all $\gamma^\vee \in \Delta^{\vee +}_{r-1}$. Since these latter entries are scalars they equal their leading terms.  Let $P^0$ defined by replacing \textit{all} the entries of $P$ by their leading terms.

%Let $P_\gamma$ be such a non-trivial linear combination.

By Remark 2 of \ref {8.1}, the entries $P^0$ are obtained by applying products of the $A_i:i\in I$ to the corresponding non-trivial copy $\{\sum_{j=1}^{n_{r-1}-n_r}d_j\partial q_{(j)}/\partial \varpi_i\}_{i \in I}$ of the adjoint module in $Q$ and in degree $r-1$.  In particular the entries for which $o(\gamma^\vee)=r-1$ are obtained by applying the $A \in \textbf {A}$ of length $r-1$ (starting from $A_i$ on $\sum_{j=1}^{n_{r-1}-n_r}d_j\partial q_{(j)}/\partial \varpi_i$).  Then by Lemma \ref {B.3} these entries cannot all vanish and this contradiction proves the required linear independence.

\end {proof}
\subsection{}\label{8.4}

View $M_r$ as a linear transformation of $X_{r-1}$ onto $X_r$, where $X_r:r \in \mathbb N^+$ is the $n_r$ dimensional $k$-vector space $\oplus_{\gamma^\vee \in \Delta^{\vee +}_r}kz_{\gamma^\vee}$.  Choose a subset $R \subset \Delta^{\vee +}_{r-1}$ such that the $z_{\gamma^\vee}:\gamma^\vee \in S:=\Delta^{\vee +}_{r-1}\setminus R$ is a basis for $X_{r-1}/ \ker M_r$.  Let $Y_{r-1}$(resp. $Z_{r-1}$) be the linear span of the $z_{\gamma^\vee}$:$\gamma^\vee \in S$ (resp. $R$).  By construction $Z_{r-1}=\ker M_r \quad \text {mod} \ Y_{r-1}$.  Then by 
Lemma \ref {8.3} the matrix with entries  $P_{(j),\gamma^\vee}:j=1,2,\ldots,n_{r-1}-n_r, \gamma^\vee \in R$ has non-zero determinant. Consequently for any set $(z^0_{\gamma^\vee})_{\gamma^\vee \in R}$ of scalars we can find a linear
combination $(\hat{P}_{\gamma^\vee})_{\gamma^\vee \in \Delta^{\vee +}}=\sum_{j=1}^{n_{r-1}-n_r}b_j(P_{(j),\gamma^\vee})_{\gamma^\vee \in \Delta^{\vee+}}$ of
 $|\Delta^{\vee +}|$-tuples, so that $\hat{P}_{\gamma^\vee} =z^0_{\gamma^\vee}, \forall \gamma^\vee \in R$.

By our induction hypothesis, $c_r$ on the left hand side of $(24)$ is defined. Through the above we can modify $(P_{\gamma^\vee})_{\gamma^\vee \in \Delta^{\vee +}}$ by adding to it an appropriate choice of $(\hat{P}_{\gamma^\vee})_{\gamma^\vee \in \Delta^{\vee +}}$ so that the $P_{\gamma^\vee}(s\rho):\gamma^\vee \in R$ satisfy the condition imposed by the right hand side of $(24)$ on the entries $\gamma^\vee \in R$. Moreover these latter terms come from Zhelobenko invariants of strictly lower degree.  Finally, just as in the case for which $\ker M_r=0$, by comparing the first parts of Propositions \ref {7.8} and \ref {8.1}, it follows that the remaining entries must also satisfy $(24)$.   This proves Proposition \ref {8.2}.
 
 %noting that the $P_{(j),i}:i \in I$ have common degree $r-2$ and since $r\leq m$, this is strictly less than the common degree $m-1$ of the $P_i:i \in I$.

\subsection{}\label{8.5}

 By Proposition \ref {8.2} it is possible to satisfy the condition in \ref {3.6} by adding to the $P_i$ strictly lower order terms of the same form (that is to say coming from Zhelobenko invariants of strictly lower degree). Then by induction on degree we obtain from Corollary \ref {3.5} a proof of the analogue Kostant conjecture outside type $G_2$. Yet as noted in \ref {C.2} the latter case being of rank $2$ is trivial.  Thus we have proved the following

\begin {thm}  Let $J:=\sum_{i \in I}\varpi_i\otimes q_i$ be a Zhelobenko invariant and let $m$ be the common degree of the $q_i$.  Then
$$(e^\vee)^{m+1}(\sum_{i \in I}\varpi_iq_i(s\rho)=0), \forall s \in k.$$
\end {thm}

\textbf{Remark} 1.  There is one minor point we should mention.  In passing from the $q_i$ to the $P_i$ we not only divided by $h_i+2$ (which all take the same value on $s\rho$); but we also (in \ref {2.5}) applied the automorphism $\theta$ which makes a translation of argument by $\rho$.  Of course this only alters $s$ and so does not affect the result.

\

\textbf{Remark} 2. There are thus \textit{three} reasons which make evaluation at multiples of $\rho$ special with regard to the (analogue) Kostant conjecture.   The first is the Alekseev-Rohr proposition described in \ref {3.2}.  The second is the relationship described in the comparison of Propositions \ref {7.8} and \ref {8.1} culminating in Eq. $(24)$.  The third is that pointed out in Remark 1 above.

\section{The Figures}

The Zhelobenko monoid is illustrated in rank $4$.

In Figures $1-3$ the Zelobenko monoid in the non-simply laced cases $B_4,C_4,F_4$ is presented as a graph. Some additional edges are included to describe the action of bad pairs. (In type $F_4$ we were unable to present the graph in a planar or even three dimensional fashion.  For this reason dotted lines were used to represent edges out of the plane.  The result resembles one of M. C. Escher's impossible three dimensional figures.) In addition a horizontal broken line with label $s$ joins the vertices $P_{\gamma^\vee}:o(\gamma^\vee)=s$.

The vertex labelled $(i)$ on the bottom row corresponds to the element $P_i:i=1,2,3,4$.  An edge labelled $j$  corresponds to the element $A_j:j=1,2,3,4$.  It joins a vertex labelled by $P$ to a vertex labelled by $A_jP$. Vertices with unencircled labels describe the image of
$\mathscr P$.  Edges with unencircled labels (resp. unlabelled edges) joining unencircled vertices describe the transition $P_{\gamma^\vee} \rightarrow P_{\delta^\vee}$ effected by a good (resp. bad) pair when $o(\delta^\vee)-o(\gamma^\vee)=1$ (resp. $2$).

In Figures $4,5$, the simply-laced cases $A_4,D_4$ were drawn for comparison.  In this case $\mathscr P$ is bijective so in particular there are no encircled vertices.  Again all joined vertices correspond to good pairs and so there are no encircled labels on vertices or on edges and no added edges.

Finally Figure $3^*$ describes some data which was technically impossible to include in Figure $3$. Specifically the images of the map $\mathscr P$ is described. Here the $\alpha_1^\vee,\alpha_2^\vee$ are taken to be \textit{long} coroots which is the \textit{opposite} convention to that of Bourbaki). It allows us to use \cite [Planche VIII]{B} as if these were roots. Thus for example $1121$ means the \textit{coroot} $\alpha_1^\vee + \alpha_2^\vee + 2 \alpha_3^\vee + \alpha_4^\vee$.

In Figure $3^*$ all edges between unencircled edges, \textit{except those carrying a shaded circle}, correspond to a transition allowed by a good or a bad pair (see \ref {7.3}).  This figure can be used to illustrate the truth of Lemma \ref {7.7}(i) in type $F_4$ (which nevertheless has a case by case free proof.)

\begin{center}

\begin{picture}(300,400)(-110,0)
%\setlength{\unitlength}{0.5}
%\put(80,-50){\line(0,1){480}}
%\multiput(73,160)(6,0){8}{\line(1,0){3}}

\multiput(-130,400)(7,0){53}{\line(1,0){2}}
\put(240,397){$7$}
\multiput(-130,340)(7,0){53}{\line(1,0){2}}
\put(240,337){$6$}
\multiput(-130,280)(7,0){53}{\line(1,0){2}}
\put(240,277){$5$}
\multiput(-130,220)(7,0){53}{\line(1,0){2}}
\put(240,217){$4$}
\multiput(-130,160)(7,0){53}{\line(1,0){2}}
\put(240,157){$3$}
\multiput(-130,100)(7,0){53}{\line(1,0){2}}
\put(240,97){$2$}
\multiput(-130,40)(7,0){53}{\line(1,0){2}}
\put(240,37){$1$}

%\put(-130,400){\line(1,0){330}}
%\put(-130,340){\line(1,0){330}}
%\put(-130,280){\line(1,0){330}}
%\put(-130,220){\line(1,0){330}}
%\put(-130,160){\line(1,0){330}}
%\put(-130,100){\line(1,0){330}}
%\put(-130,40){\line(1,0){330}}

%\linethickness{0.4mm}

\put(80,400){\circle*{4}}
\put(130,340){\circle*{4}}
\put(140,280){\circle*{4}}

\put(0,340){\circle*{4}}
\put(0,280){\circle*{4}}

\put(0,340){\line(4,3){80}}
\put(130,340){\line(-4,5){48}}
\put(0,280){\line(0,1){60}}
\put(140,280){\line(-1,5){12}}

\put(-40,280){\circle*{4}}
\put(-40,220){\circle*{4}}
\put(-80,220){\circle*{4}}
\put(60,220){\circle*{4}}
\put(140,220){\circle*{4}}

\put(-80,220){\line(2,3){80}}
\put(-40,220){\line(0,1){60}}
\put(-80,160){\line(2,3){80}}

\put(60,220){\line(-1,1){60}}
\put(60,220){\line(4,3){80}}
\put(140,220){\line(0,1){60}}

\put(-30,160){\circle*{4}}
\put(-80,160){\circle*{4}}
\put(60,160){\circle*{4}}
\put(140,160){\circle*{4}}

%\put(60,220){\line(-1,1){60}}
\put(-30,160){\line(-1,6){10}}
\put(-80,160){\line(0,1){60}}
\put(60,40){\line(0,1){180}}
\put(140,100){\line(0,1){120}}
\put(60,160){\line(4,3){80}}

\put(-50,100){\circle*{4}}
\put(30,100){\circle*{4}}
\put(60,100){\circle*{4}}
\put(140,100){\circle*{4}}

%\put(-50,100){\line(-1,2){30}}
\put(-50,100){\line(1,3){20}}
\put(30,100){\line(-1,1){60}}
\put(30,100){\line(1,2){30}}
\put(60,100){\line(4,3){80}}
\put(130,40){\line(1,6){10}}
\put(-50,100){\line(-1,2){30}}
\put(-10,40){\line(-2,3){40}}

\put(-50,100){\circle*{4}}
\put(30,100){\circle*{4}}
\put(60,100){\circle*{4}}
%\put(130,100){\circle*{4}}

\put(-10,40){\circle*{4}}
\put(30,40){\circle*{4}}
\put(60,40){\circle*{4}}
\put(130,40){\circle*{4}}

\put(60,40){\line(4,3){80}}
\put(30,40){\line(0,1){60}}
\put(30,40){\line(1,2){30}}

\put(108,370){$1$}
\put(32,370){$2$}
\put(111,374){\circle{12}}

\put(135,310){$2$}
\put(-28,310){$3$}
\put(0,310){$1$}

\put(-67,250){$4$}
\put(-40,250){$1$}
\put(-15,250){$3$}
\put(32,250){$2$}
\put(95,250){$1$}
\put(142,250){$3$}
\put(35,254){\circle{12}}

\put(-68,190){$4$}
\put(-88,190){$1$}
\put(-35,190){$2$}
\put(53,190){$3$}
\put(95,190){$1$}
\put(142,190){$4$}

\put(-48,130){$4$}
\put(-75,130){$2$}
\put(0,130){$3$}
\put(36,130){$2$}
\put(95,130){$1$}
\put(142,130){$3$}
\put(60,130){$4$}
\put(3,134){\circle{12}}

\put(-47,70){$3$}
\put(22,70){$4$}
\put(40,70){$2$}
\put(60,70){$3$}
\put(96,70){$1$}
\put(136,70){$2$}

\put(-27,25){$(4)$}
\put(23,25){$(3)$}
\put(53,25){$(2)$}
\put(123,25){$(1)$}

\put(0,340){\circle{9}}
\put(-40,280){\circle{9}}

\put(-80,220){\circle{9}}
\put(-40,220){\circle{9}}
\put(-80,160){\circle{9}}
\put(-50,100){\circle{9}}

%\multiput(-30,160)(6,0){13}{\dottedline(1,5){3}}
%\linethickness{0.1mm}

\put(-30,160){\line(1,4){30}}
\put(-10,40){\line(-1,6){20}}
\put(0,280){\line(2,3){80}}

%\put(127,157){$c_1$}
%\multiput(-26,150)(6,0){13}{\line(1,0){3}}
%\put(-70,147){$\varphi(t_{2j-1})$}
%\put(55,147){$A$}
%\multiput(-26,50)(6,0){13}{\line(1,0){3}}
%\put(-62,47){$\varphi(t_{2j})$}
%\put(55,-14){$A$}
%\multiput(-26,-10)(6,0){13}{\line(1,0){3}}
%\put(-70,-13){$\varphi(t_{2j+1})$}
%\put(55,46){$B$}
%\put(50,140){\circle{7}}
%\put(50,20){\circle{7}}
%\multiput(93,40)(6,0){5}{\line(1,0){3}}
%\put(127,37){$c_2$}
%\linethickness{0.4mm}
%\put(10,160){\line(1,0){60}}
%\put(30,140){\line(1,0){60}}
%\put(30,120){\line(1,0){40}}
%\put(30,100){\line(1,0){40}}
%\put(30,80){\line(1,0){40}}
%\put(10,60){\line(1,0){60}}
%\put(30,40){\line(1,0){60}}
%\put(30,20){\line(1,0){40}}
%\put(30,0){\line(1,0){40}}
%%
%\put(10,60){\line(0,1){100}}
%\put(30,120){\line(0,1){20}}
%\put(30,80){\line(0,1){20}}
%\put(30,20){\line(0,1){20}}

%%
%\put(70,100){\line(0,1){20}}
%\put(70,60){\line(0,1){20}}
%\put(70,00){\line(0,1){20}}
%\put(90,40){\line(0,1){100}}
%\multiput(50,-10)(0,20){9}{\circle*{3}}
%%
%\multiput(70,160)(0,4){6}{\line(0,1){2}}
%\multiput(30,-22)(0,4){6}{\line(0,1){2}}
\end{picture}

\end{center}
\begin{center}
{\it Figure $1$.} \\
\end{center}

\begin{center}
{\bf The Zhelobenko monoid in type $B_4$.} \\
\end{center}

\begin{center}

\begin{picture}(300,400)(-110,0)
%\setlength{\unitlength}{0.5}
%\put(80,-50){\line(0,1){480}}
%\multiput(73,160)(6,0){8}{\line(1,0){3}}

%\multiput(-50,-50)(10,0){8}{\line(1,0){7}}

\multiput(-130,400)(7,0){53}{\line(1,0){2}}
\put(240,397){$7$}
\multiput(-130,340)(7,0){53}{\line(1,0){2}}
\put(240,337){$6$}
\multiput(-130,280)(7,0){53}{\line(1,0){2}}
\put(240,277){$5$}
\multiput(-130,220)(7,0){53}{\line(1,0){2}}
\put(240,217){$4$}
\multiput(-130,160)(7,0){53}{\line(1,0){2}}
\put(240,157){$3$}
\multiput(-130,100)(7,0){53}{\line(1,0){2}}
\put(240,97){$2$}
\multiput(-130,40)(7,0){53}{\line(1,0){2}}
\put(240,37){$1$}

%\put(-130,400){\line(1,0){330}}
%\put(-130,340){\line(1,0){330}}
%\put(-130,280){\line(1,0){330}}
%\put(-130,220){\line(1,0){330}}
%\put(-130,160){\line(1,0){330}}
%\put(-130,100){\line(1,0){330}}
%\put(-130,40){\line(1,0){330}}

\put(80,400){\circle*{4}}
\put(130,340){\circle*{4}}
\put(140,280){\circle*{4}}

\put(130,340){\circle{9}}
\put(140,280){\circle{9}}

\put(0,340){\circle*{4}}
\put(0,280){\circle*{4}}

\put(0,340){\line(4,3){80}}
\put(130,340){\line(-4,5){48}}
\put(0,280){\line(0,1){60}}
\put(140,280){\line(-1,5){12}}

\put(-40,280){\circle*{4}}
\put(-40,220){\circle*{4}}
\put(-80,220){\circle*{4}}
\put(60,220){\circle*{4}}
\put(140,220){\circle*{4}}

\put(60,220){\circle{9}}
\put(140,220){\circle{9}}

\put(-80,220){\line(2,3){80}}
\put(-40,220){\line(0,1){60}}
\put(-80,160){\line(2,3){80}}

\put(60,220){\line(-1,1){60}}
\put(60,220){\line(4,3){80}}
\put(140,220){\line(0,1){60}}

\put(-30,160){\circle*{4}}
\put(-80,160){\circle*{4}}
\put(60,160){\circle*{4}}
\put(140,160){\circle*{4}}

\put(60,160){\circle{9}}

%\put(60,220){\line(-1,1){60}}
\put(-30,160){\line(-1,6){10}}
\put(-80,160){\line(0,1){60}}
\put(60,40){\line(0,1){180}}
\put(140,100){\line(0,1){120}}
\put(60,160){\line(4,3){80}}

\put(-50,100){\circle*{4}}
\put(30,100){\circle*{4}}
\put(60,100){\circle*{4}}
\put(140,100){\circle*{4}}

\put(30,100){\circle{9}}

%\put(-50,100){\line(-1,2){30}}
\put(-50,100){\line(1,3){20}}
\put(30,100){\line(-1,1){60}}
\put(30,100){\line(1,2){30}}
\put(60,100){\line(4,3){80}}
\put(130,40){\line(1,6){10}}
\put(-20,40){\line(-1,2){60}}

\put(-50,100){\circle*{4}}
\put(30,100){\circle*{4}}
\put(60,100){\circle*{4}}
%\put(-44,134){\circle{12}}

\put(-20,40){\circle*{4}}
\put(30,40){\circle*{4}}
\put(60,40){\circle*{4}}
\put(130,40){\circle*{4}}
\put(30,40){\line (-1,2){60}}

\put(60,40){\line(4,3){80}}
\put(30,40){\line(0,1){60}}
\put(30,40){\line(1,2){30}}

\put(108,370){$1$}
\put(32,370){$2$}

\put(135,310){$2$}
\put(-28,310){$3$}
\put(0,310){$1$}

\put(-67,250){$4$}
\put(-40,250){$1$}
\put(-15,250){$3$}
\put(32,250){$2$}
\put(95,250){$1$}
\put(142,250){$3$}
\put(-64,254){\circle{12}}

\put(-65,194){\circle{12}}
\put(-68,190){$4$}
\put(-88,190){$1$}
\put(-35,190){$2$}
\put(53,190){$3$}
\put(95,190){$1$}
\put(142,190){$4$}
\put(-44,134){\circle{12}}

\put(-48,130){$4$}
\put(-75,130){$2$}
\put(0,130){$3$}
\put(36,130){$2$}
\put(95,130){$1$}
\put(142,130){$3$}
\put(60,130){$4$}

\put(-47,70){$3$}
\put(22,70){$4$}
\put(40,70){$2$}
\put(60,70){$3$}
\put(96,70){$1$}
\put(136,70){$2$}

\put(-27,25){$(4)$}
\put(23,25){$(3)$}
\put(53,25){$(2)$}
\put(123,25){$(1)$}

%\put(127,157){$c_1$}
%\multiput(-26,150)(6,0){13}{\line(1,0){3}}
%\put(-70,147){$\varphi(t_{2j-1})$}
%\put(55,147){$A$}
%\multiput(-26,50)(6,0){13}{\line(1,0){3}}
%\put(-62,47){$\varphi(t_{2j})$}
%\put(55,-14){$A$}
%\multiput(-26,-10)(6,0){13}{\line(1,0){3}}
%\put(-70,-13){$\varphi(t_{2j+1})$}
%\put(55,46){$B$}
%\put(50,140){\circle{7}}
%\put(50,20){\circle{7}}
%\multiput(93,40)(6,0){5}{\line(1,0){3}}
%\put(127,37){$c_2$}
%\linethickness{0.4mm}
%\put(10,160){\line(1,0){60}}
%\put(30,140){\line(1,0){60}}
%\put(30,120){\line(1,0){40}}
%\put(30,100){\line(1,0){40}}
%\put(30,80){\line(1,0){40}}
%\put(10,60){\line(1,0){60}}
%\put(30,40){\line(1,0){60}}
%\put(30,20){\line(1,0){40}}
%\put(30,0){\line(1,0){40}}
%%
%\put(10,60){\line(0,1){100}}
%\put(30,120){\line(0,1){20}}
%\put(30,80){\line(0,1){20}}
%\put(30,20){\line(0,1){20}}

%%
%\put(70,100){\line(0,1){20}}
%\put(70,60){\line(0,1){20}}
%\put(70,00){\line(0,1){20}}
%\put(90,40){\line(0,1){100}}
%\multiput(50,-10)(0,20){9}{\circle*{3}}
%%
%\multiput(70,160)(0,4){6}{\line(0,1){2}}
%\multiput(30,-22)(0,4){6}{\line(0,1){2}}
\end{picture}

\end{center}
\begin{center}
{\it Figure $2$.} \\
\end{center}

\begin{center}
{\bf The Zhelobenko monoid in type $C_4$.} \\
\end{center}

\begin{center}

\begin{picture}(300,500)(-110,20)
\setlength{\unitlength}{0.28mm}
%\put(80,-50){\line(0,1){480}}
%\multiput(73,160)(6,0){8}{\line(1,0){3}}

%\multiput(-50,-50)(10,0){8}{\line(1,0){7}}

\multiput(-170,640)(7,0){68}{\line(1,0){2}}
\put(305,637){$11$}
\multiput(-170,580)(7,0){68}{\line(1,0){2}}
\put(305,577){$10$}
\multiput(-170,520)(7,0){68}{\line(1,0){2}}
\put(305,517){$9$}
\multiput(-170,460)(7,0){68}{\line(1,0){2}}
\put(305,457){$8$}
\multiput(-170,400)(7,0){68}{\line(1,0){2}}
\put(305,397){$7$}
\multiput(-170,340)(7,0){68}{\line(1,0){2}}
\put(305,337){$6$}
\multiput(-170,280)(7,0){68}{\line(1,0){2}}
\put(305,277){$5$}
\multiput(-170,220)(7,0){68}{\line(1,0){2}}
\put(305,217){$4$}
\multiput(-170,160)(7,0){68}{\line(1,0){2}}
\put(305,157){$3$}
\multiput(-170,100)(7,0){68}{\line(1,0){2}}
\put(305,97){$2$}
\multiput(-170,40)(7,0){68}{\line(1,0){2}}
\put(305,37){$1$}

%\put(-130,400){\line(1,0){330}}
%\put(-130,340){\line(1,0){330}}
%\put(-130,280){\line(1,0){330}}
%\put(-130,220){\line(1,0){330}}
%\put(-130,160){\line(1,0){330}}
%\put(-130,100){\line(1,0){330}}
%\put(-130,40){\line(1,0){330}}

\put(50,640){\circle*{4}}
\put(-10,580){\circle*{4}}
\put(110,580){\circle*{4}}
\put(110,580){\circle{9}}

\put(50,520){\circle*{4}}
\put(-70,520){\circle*{4}}
\put(170,520){\circle*{4}}
\put(50,520){\circle{9}}
\put(170,520){\circle{9}}

\put(110,580){\circle{9}}

\put(-130,460){\circle*{4}}
\put(-10,460){\circle*{4}}
\put(110,460){\circle*{4}}
\put(230,460){\circle*{4}}
\put(-10,460){\circle{9}}
\put(110,460){\circle{9}}
\put(230,460){\circle{9}}

\put(230,460){\line(-1,1){180}}
\put(-130,460){\line(1,1){180}}
\put(110,460){\line(-1,1){120}}
\put(-10,460){\line(1,1){120}}
\put(110,460){\line(1,1){60}}
\put(-10,460){\line(-1,1){60}}
\put(-30,400){\line (-1,3){40}}

\put(-150,400){\circle*{4}}
\put(-30,400){\circle*{4}}
\put(130,400){\circle*{4}}
\put(250,400){\circle*{4}}
\put(130,400){\circle{9}}
\put(250,400){\circle{9}}

\put(130,400){\line(-1,3){20}}
\put(-150,400){\line(1,3){20}}
\put(-30,400){\line(1,3){20}}
\put(250,400){\line(-1,3){20}}

\put(-90,340){\circle*{4}}
\put(-10,340){\circle*{4}}
\put(110,340){\circle*{4}}
\put(190,340){\circle*{4}}
\put(110,340){\circle{9}}
\put(190,340){\circle{9}}

\put(-90,340){\line(-1,1){60}}
\put(-10,340){\line(-1,3){20}}
\put(110,340){\line(1,3){20}}
\put(190,340){\line(1,1){60}}
\put(-90,340){\line(1,1){60}}
\put(190,340){\line(-1,1){60}}

\put(-70,280){\circle*{4}}
\put(-50,280){\circle*{4}}
\put(50,280){\circle*{4}}
\put(150,280){\circle*{4}}
\put(170,280){\circle*{4}}
\put(50,280){\circle{9}}
\put(170,280){\circle{9}}

\put(-70,280){\line(-1,3){20}}
\put(-70,280){\line(1,1){60}}
\put(-50,280){\line(2,3){40}}
\put(50,280){\line(-1,1){60}}
\put(50,280){\line(1,1){60}}
\put(150,280){\line(-2,3){40}}
\put(170,280){\line(-1,1){60}}
\put(170,280){\line(1,3){20}}

\put(-90,220){\circle*{4}}
\put(-30,220){\circle*{4}}
\put(90,220){\circle*{4}}
\put(130,220){\circle*{4}}
\put(10,220){\circle*{4}}
\put(190,220){\circle*{4}}
\put(10,220){\circle{9}}
\put(130,220){\circle{9}}
\put(190,220){\circle{9}}

\put(-90,220){\line(1,3){20}}
%\put(-30,220){\line(-4,6){40}}
\put(10,220){\line(-1,1){60}}
\put(10,220){\line(2,3){40}}
\put(90,220){\line(-2,3){40}}
\put(90,220){\line(1,1){60}}
\put(190,220){\line(-1,3){20}}
%\put(130,220){\line(4,6){40}}
\put(190,220){\line(-2,3){40}}
\put(-90,220){\line(2,3){40}}
\multiput(130,220)(2,3){20}{\circle*{1}}
\multiput(-30,220)(-2,3){20}{\circle*{1}}

\put(-50,160){\circle*{4}}
\put(30,160){\circle*{4}}
\put(70,160){\circle*{4}}
\put(150,160){\circle*{4}}
\put(50,160){\circle*{4}}
\put(70,160){\circle{9}}
\put(150,160){\circle{9}}

\put(-50,160){\line(-2,3){40}}
\put(150,160){\line(2,3){40}}
%\put(-50,160){\line(2,6){20}}
%\put(150,160){\line(-2,6){20}}
%\put(30,160){\line(-2,2){60}}
%\put(70,160){\line(2,2){60}}
%\put(70,160){\line(-2,2){60}}
%\put(30,160){\line(2,2){60}}
\multiput(-50,160)(1,3){20}{\circle*{1}}
\multiput(150,160)(-1,3){20}{\circle*{1}}
\multiput(30,160)(-3,3){20}{\circle*{1}}
\multiput(70,160)(3,3){20}{\circle*{1}}
\multiput(70,160)(-3,3){20}{\circle*{1}}
\multiput(30,160)(3,3){20}{\circle*{1}}

\put(-70,100){\circle*{4}}
\put(10,100){\circle*{4}}
\put(90,100){\circle*{4}}
\put(170,100){\circle*{4}}
\put(90,100){\circle{9}}

\put(50,160){\line(-2,3){40}}
\put(50,160){\line(2,3){40}}

%\put(90,100){\line(-2,6){20}}
\put(-70,100){\line(1,3){20}}
%\put(10,100){\line(2,6){20}}
\put(170,100){\line(-1,3){20}}
\multiput(90,100)(-1,3){20}{\circle*{1}}
\multiput(10,100)(1,3){20}{\circle*{1}}

\put(10,100){\line(2,3){40}}
\put(90,100){\line(-2,3){40}}
\put(10,100){\line(-1,1){60}}
\put(90,100){\line(1,1){60}}

\put(-90,40){\circle*{4}}
\put(-10,40){\circle*{4}}
\put(110,40){\circle*{4}}
\put(190,40){\circle*{4}}

\put(-90,40){\line(1,3){20}}
\put(-10,40){\line(1,3){20}}
\put(-10,40){\line(-1,1){60}}
\put(110,40){\line(-1,3){20}}
\put(110,40){\line(1,1){60}}
\put(190,40){\line(-1,3){20}}

\put(16,610){$1$}
\put(79,610){$4$}
\put(-48,550){$2$}
\put(76,550){$1$}
\put(20,550){$4$}
\put(140,550){$3$}

\put(16,610){$1$}
\put(79,610){$4$}

\put(-107,490){$3$}
\put(-40,490){$4$}
\put(28,490){$2$}
\put(67,490){$3$}
\put(135,490){$1$}
\put(200,490){$2$}
\put(-104,495){\circle{11}}

\put(-146,430){$4$}
\put(-17,430){$3$}
\put(110,430){$2$}
\put(240,430){$1$}

\put(-134,370){$3$}
\put(-53,370){$4$}
\put(-20,370){$2$}
\put(114,370){$3$}
\put(145,370){$1$}
\put(228,370){$2$}
\put(-49,373){\circle{11}}

\put(-90,310){$2$}
\put(-26,310){$1$}
\put(-48,310){$4$}
\put(21,310){$3$}
\put(72,310){$2$}
\put(117,310){$4$}
\put(140,310){$1$}
\put(183,310){$3$}
\put(-45,314){\circle{11}}

\put(180,250){$4$}
\put(166,240){$1$}
\put(-48,250){$3$}
\put(-21,250){$3$}
\put(24,250){$1$}
\put(69,250){$4$}
\put(116,250){$2$}
\put(143,250){$2$}
\put(-85,250){$1$}
\put(-72,240){$4$}
\put(-68,244){\circle{11}}

\put(-80,190){$3$}
\put(-45,190){$1$}
\put(-15,190){$4$}
\put(35,180){$4$}
\put(36,195){$2$}
\put(60,195){$3$}
\put(74,190){$1$}
\put(109,190){$1$}
\put(130,190){$4$}
\put(176,190){$2$}
\put(-68,244){\circle{11}}
\put(64,199){\circle{11}}

\put(-65,130){$2$}
\put(-21,130){$4$}
\put(15,130){$1$}
\put(33,130){$3$}
\put(66,130){$2$}
\put(80,130){$4$}
\put(115,130){$1$}
\put(160,130){$3$}
\put(36,134){\circle{11}}

\put(-88,70){$3$}
\put(-41,70){$4$}
\put(2,70){$2$}
\put(96,70){$3$}
\put(136,70){$1$}
\put(180,70){$2$}

\put(-98,27){$(4)$}
\put(-18,27){$(3)$}

\put(102,27){$(2)$}
\put(182,27){$(1)$}

\put(110,40){\line(-1,2){60}}
\put(50,160){\line(-5,6){100}}

\end{picture}

\end{center}
\begin{center}
{\it Figure $3$.} \\
\end{center}

\begin {center}
{\bf The Zhelobenko monoid in type $F_4$.}\\
\end{center}

\begin{center}

\begin{picture}(300,500)(-110,20)
\setlength{\unitlength}{0.28mm}

\multiput(-170,640)(7,0){68}{\line(1,0){2}}
\put(305,637){$11$}
\multiput(-170,580)(7,0){68}{\line(1,0){2}}
\put(305,577){$10$}
\multiput(-170,520)(7,0){68}{\line(1,0){2}}
\put(305,517){$9$}
\multiput(-170,460)(7,0){68}{\line(1,0){2}}
\put(305,457){$8$}
\multiput(-170,400)(7,0){68}{\line(1,0){2}}
\put(305,397){$7$}
\multiput(-170,340)(7,0){68}{\line(1,0){2}}
\put(305,337){$6$}
\multiput(-170,280)(7,0){68}{\line(1,0){2}}
\put(305,277){$5$}
\multiput(-170,220)(7,0){68}{\line(1,0){2}}
\put(305,217){$4$}
\multiput(-170,160)(7,0){68}{\line(1,0){2}}
\put(305,157){$3$}
\multiput(-170,100)(7,0){68}{\line(1,0){2}}
\put(305,97){$2$}
\multiput(-170,40)(7,0){68}{\line(1,0){2}}
\put(305,37){$1$}

\put(50,640){\circle*{4}}
\put(-10,580){\circle*{4}}
\put(110,580){\circle*{4}}
\put(110,580){\circle{9}}

\put(40,645){$2342$}

\put(0,581){$1342$}

\put(50,520){\circle*{4}}
\put(-70,520){\circle*{4}}
\put(170,520){\circle*{4}}
\put(50,520){\circle{9}}
\put(170,520){\circle{9}}

\put(-60,521){$1242$}

\put(110,580){\circle{9}}

\put(-130,460){\circle*{4}}
\put(-10,460){\circle*{4}}
\put(110,460){\circle*{4}}
\put(230,460){\circle*{4}}
\put(-10,460){\circle{9}}
\put(110,460){\circle{9}}
\put(230,460){\circle{9}}

\put(-120,461){$1232$}

\put(230,460){\line(-1,1){180}}
\put(-130,460){\line(1,1){180}}
\put(110,460){\line(-1,1){120}}
\put(-10,460){\line(1,1){120}}
\put(110,460){\line(1,1){60}}
\put(-10,460){\line(-1,1){60}}
\put(-30,400){\line (-1,3){40}}

\put(-150,400){\circle*{4}}
\put(-30,400){\circle*{4}}
\put(130,400){\circle*{4}}
\put(250,400){\circle*{4}}
\put(130,400){\circle{9}}
\put(250,400){\circle{9}}

\put(-145,401){$1231$}
\put(-25,401){$1222$}
%\put(130,400){\circle*{4}}
%\put(250,400){\circle*{4}}

\put(130,400){\line(-1,3){20}}
\put(-150,400){\line(1,3){20}}
\put(-30,400){\line(1,3){20}}
\put(250,400){\line(-1,3){20}}

\put(-90,340){\circle*{4}}
\put(-10,340){\circle*{4}}
\put(110,340){\circle*{4}}
\put(190,340){\circle*{4}}
\put(110,340){\circle{9}}
\put(190,340){\circle{9}}

\put(-80,341){$1221$}
\put(-10,341){$1122$}

\put(-90,340){\line(-1,1){60}}
\put(-10,340){\line(-1,3){20}}
\put(110,340){\line(1,3){20}}
\put(190,340){\line(1,1){60}}
\put(-90,340){\line(1,1){60}}
\put(190,340){\line(-1,1){60}}

\put(-70,280){\circle*{4}}
\put(-50,280){\circle*{4}}
\put(50,280){\circle*{4}}
\put(150,280){\circle*{4}}
\put(170,280){\circle*{4}}
\put(50,280){\circle{9}}
\put(170,280){\circle{9}}

\put(-105,281){$1121$}
\put(-42,281){$0122$}
%\put(50,280){\circle*{4}}
\put(115,281){$1220$}
%\put(170,280){\circle*{4}}

\put(-70,280){\line(-1,3){20}}
\put(-70,280){\line(1,1){60}}
\put(-50,280){\line(2,3){40}}
\put(50,280){\line(-1,1){60}}
\put(50,280){\line(1,1){60}}
\put(150,280){\line(-2,3){40}}
\put(170,280){\line(-1,1){60}}
\put(170,280){\line(1,3){20}}

\put(-90,220){\circle*{4}}
\put(-30,220){\circle*{4}}
\put(90,220){\circle*{4}}
\put(130,220){\circle*{4}}
\put(10,220){\circle*{4}}
\put(190,220){\circle*{4}}
\put(10,220){\circle{9}}
\put(130,220){\circle{9}}
\put(190,220){\circle{9}}

\put(-120,221){$0121$}
\put(-45,221){$1111$}
\put(75,221){$1120$}

\put(-90,220){\line(1,3){20}}
%\put(-30,220){\line(-4,6){40}}
\put(10,220){\line(-1,1){60}}
\put(10,220){\line(2,3){40}}
\put(90,220){\line(-2,3){40}}
\put(90,220){\line(1,1){60}}
\put(190,220){\line(-1,3){20}}

\put(190,220){\line(-2,3){40}}
\put(-90,220){\line(2,3){40}}
\multiput(130,220)(2,3){20}{\circle*{1}}
\multiput(-30,220)(-2,3){20}{\circle*{1}}

\put(-50,160){\circle*{4}}
\put(30,160){\circle*{4}}
\put(70,160){\circle*{4}}
\put(150,160){\circle*{4}}
\put(50,160){\circle*{4}}
\put(70,160){\circle{9}}
\put(150,160){\circle{9}}

\put(-85,161){$0111$}
\put(-10,161){$1110$}
\put(35,149){$0120$}

\put(-50,160){\line(-2,3){40}}
\put(150,160){\line(2,3){40}}

\multiput(-50,160)(1,3){20}{\circle*{1}}
\multiput(150,160)(-1,3){20}{\circle*{1}}
\multiput(30,160)(-3,3){20}{\circle*{1}}
\multiput(70,160)(3,3){20}{\circle*{1}}
\multiput(70,160)(-3,3){20}{\circle*{1}}
\multiput(30,160)(3,3){20}{\circle*{1}}

\put(-70,100){\circle*{4}}
\put(10,100){\circle*{4}}
\put(90,100){\circle*{4}}
\put(170,100){\circle*{4}}
\put(90,100){\circle{9}}

\put(-100,101){$0011$}
\put(-24,101){$0110$}
%\put(90,100){\circle*{4}}
\put(170,101){$1100$}

\put(50,160){\line(-2,3){40}}
\put(50,160){\line(2,3){40}}

%\put(90,100){\line(-2,6){20}}
\put(-70,100){\line(1,3){20}}
%\put(10,100){\line(2,6){20}}
\put(170,100){\line(-1,3){20}}
\multiput(90,100)(-1,3){20}{\circle*{1}}
\multiput(10,100)(1,3){20}{\circle*{1}}

\put(10,100){\line(2,3){40}}
\put(90,100){\line(-2,3){40}}
\put(10,100){\line(-1,1){60}}
\put(90,100){\line(1,1){60}}

\put(-90,40){\circle*{4}}
\put(-10,40){\circle*{4}}
\put(110,40){\circle*{4}}
\put(190,40){\circle*{4}}

\put(-90,40){\line(1,3){20}}
\put(-10,40){\line(1,3){20}}
\put(-10,40){\line(-1,1){60}}
\put(110,40){\line(-1,3){20}}
\put(110,40){\line(1,1){60}}
\put(190,40){\line(-1,3){20}}

\put(-100,490){\circle*{9}}

\put(-60,370){\circle*{9}}
\put(34,134){\circle*{9}}
\put(68,199){\circle*{9}}

\put(-40,310){\circle*{9}}

\put(-70,250){\circle*{9}}

\put(-98,27){$0001$}
\put(-18,27){$0010$}

\put(102,27){$0100$}
\put(182,27){$1000$}

\put(110,40){\line(-1,2){60}}
\put(50,160){\line(-5,6){100}}

\end{picture}

\end{center}
\begin{center}
{\it Figure $3^*$.} \\
\end{center}

\begin {center}
{\bf The Zhelobenko monoid in type $F_4$ describing $\mathscr P$.}\\
\end{center}

\begin{center}

\begin{picture}(300,250)(-110,10)

\multiput(-130,160)(7,0){53}{\line(1,0){2}}
\put(240,157){$4$}
\multiput(-130,120)(7,0){53}{\line(1,0){2}}
\put(240,117){$3$}
\multiput(-130,80)(7,0){53}{\line(1,0){2}}
\put(240,77){$2$}
\multiput(-130,40)(7,0){53}{\line(1,0){2}}
\put(240,37){$1$}

\put(50,160){\circle*{4}}
\put(10,120){\circle*{4}}
\put(90,120){\circle*{4}}

\put(-70,40){\circle*{4}}
\put(10,40){\circle*{4}}
\put(90,40){\circle*{4}}
\put(170,40){\circle*{4}}

\put(50,80){\circle*{4}}
\put(-30,80){\circle*{4}}
\put(130,80){\circle*{4}}

\put(-70,40){\line(1,1){120}}
\put(170,40){\line(-1,1){120}}

\put(10,40){\line(1,1){80}}
\put(10,40){\line(-1,1){40}}
\put(90,40){\line(1,1){40}}
\put(90,40){\line(-1,1){80}}

\put(25,140){$1$}
\put(70,140){$4$}

\put(-16,100){$2$}
\put(110,100){$3$}
\put(30,100){$4$}
\put(65,100){$1$}

\put(-56,60){$3$}
\put(-10,60){$4$}
\put(25,60){$2$}
\put(70,60){$3$}
\put(105,60){$1$}
\put(149,60){$2$}

\put(-79,25){$(4)$}
\put(1,25){$(3)$}
\put(83,25){$(2)$}
\put(163,25){$(1)$}

\end{picture}

\end{center}
\begin{center}
{\it Figure $4$.} \\
\end{center}

\begin{center}
{\bf The Zhelobenko monoid in type $A_4$.} \\
\end{center}

\begin{center}

\begin{picture}(300,350)(-110,10)

\multiput(-130,360)(7,0){53}{\line(1,0){2}}
\put(240,357){$5$}
\multiput(-130,280)(7,0){53}{\line(1,0){2}}
\put(240,277){$4$}
\multiput(-130,200)(7,0){53}{\line(1,0){2}}
\put(240,197){$3$}
\multiput(-130,120)(7,0){53}{\line(1,0){2}}
\put(240,117){$2$}
\multiput(-130,40)(7,0){53}{\line(1,0){2}}
\put(240,37){$1$}

\put(50,360){\circle*{4}}
\put(50,280){\circle*{4}}
\put(50,200){\circle*{4}}

\put(50,120){\circle*{4}}

\put(50,40){\circle*{4}}

\put(-30,120){\circle*{4}}

\put(-30,200){\circle*{4}}
\put(130,200){\circle*{4}}
\put(130,120){\circle*{4}}

\put(-70,40){\circle*{4}}

\put(170,40){\circle*{4}}
\put(130,40){\circle*{4}}

%\put(170,40){\circle*{4}}

%\put(50,80){\circle*{4}}
%\put(-30,80){\circle*{4}}
%\put(130,80){\circle*{4}}

\put(50,200){\line(0,1){160}}

\put(130,200){\line(-1,1){80}}
\put(-30,200){\line(1,1){80}}
\put(130,120){\line(-1,1){80}}
\put(-30,120){\line(1,1){80}}
\put(130,120){\line(0,1){80}}
\put(-30,120){\line(0,1){80}}
\put(50,40){\line(-1,1){80}}
\put(50,40){\line(1,1){80}}

\multiput(50,120)(2,2){40}{\circle*{1}}

\multiput(50,120)(-2,2){40}{\circle*{1}}

\multiput(50,40)(0,2){40}{\circle*{1}}

\put(170,40){\line(-1,2){40}}

\multiput(130,40)(-2,2){40}{\circle*{1}}

\put(-70,40){\line(1,2) {40}}

\put(43,320){$2$}
\put(43,240){$1$}

\put(4,240){$3$}
\put(91,240){$4$}

\put(-38,180){$1$}
\put(22,180){$3$}

\put(73,180){$4$}
\put(132,180){$1$}

\put(16,140){$4$}

\put(78,140){$3$}

\put(-15,90){$4$}
\put(-52,90){$2$}
\put(145,90){$2$}
\put(108,90){$3$}
\put(65,90){$2$}
\put(52,90){$1$}

\put(44,25){$(2)$}
\put(124,25){$(1)$}
\put(164,25){$(3)$}
\put(-78,25){$(4)$}

\end{picture}

\end{center}
\begin{center}
{\it Figure $5$.} \\
\end{center}

\begin{center}
{\bf The Zhelobenko monoid in type $D_4$.} \\
\end{center}

\end{document}